\documentclass[a4paper,12pt]{article}
\usepackage{amssymb}
\usepackage{color}
\usepackage{epsfig,amsfonts,amsthm,amssymb,latexsym,amsmath}
\usepackage[english]{inputenc}
\usepackage{bm}
\usepackage{mathrsfs}

\newfont{\bms}{msbm10 scaled 900}

\newfont{\bmsi}{msbm10 scaled 500}

\newfont{\bmsw}{msbm10 scaled 2400}

\newfont{\bmst}{msbm10 scaled 1600}

\theoremstyle{plain}
\newtheorem{theorem}{Theorem}[section]
\newtheorem{corollary}[theorem]{Corollary}
\newtheorem{conjecture}[theorem]{Conjecture}
\newtheorem{remark}[theorem]{Remark}
\theoremstyle{definition}

\theoremstyle{example}
\newtheorem{example}[theorem]{Example}
\theoremstyle{proposition}
\newtheorem{proposition}[theorem]{Proposition}
\theoremstyle{lemma}
\newtheorem{lemma}[theorem]{Lemma}
\newcommand{\cc}{\bm{c}}

\newcommand{\xx}{\bm{x}}
\newcommand{\yy}{\bm{y}}

\newcommand{\ee}{\bm{e}}

\begin{document}



\title{Perfect codes in the $l_p$ metric\footnotetext{Work partially supported by FAPESP grants 2014/20602-8 and 2013/25977-7 and by CNPq grant 312926/2013-8 \\ \indent \indent A. Campello is currently on leave from University of Campinas at T\'el\'ecom ParisTech, France. J. E. Strapasson and S. I. R. Costa are with University of Campinas, Brazil. G. C. Jorge is with Federal University of S\~ao Paulo, Brazil. E mails: campello@ime.unicamp.br, grasiele.jorge@unifesp.br, joao.strapasson@fca.unicamp.br, sueli@ime.unicamp.br}}

\author{A. Campello, G. C. Jorge,
        J. E. Strapasson, S. I. R. Costa
}

\date{}

\maketitle

\begin{abstract}
We investigate perfect codes in $\mathbb{Z}^n$ in the $\ell_p$ metric. Upper bounds for the packing radius $r$ of a linear perfect code in terms of the metric parameter $p$ and the dimension $n$ are derived. For $p = 2$ and $n = 2, 3$, we determine all radii for which there exist linear perfect codes. The non-existence results for codes in $\mathbb{Z}^n$ presented here imply non-existence results for codes over finite alphabets $\mathbb{Z}_q$, when the alphabet size is large enough, and have implications on some recent constructions of spherical codes.
\end{abstract}



\section{Introduction}

Let $\mathcal{S} \subset \mathbb{Z}^n$. A collection of disjoint translates of $\mathcal{S}$ is called a \textit{tiling} of $\mathbb{Z}^n$ if the union of its elements is equal to $\mathbb{Z}^n$. We investigate tilings of $\mathbb{Z}^n$ by balls in the $\ell_p$ metric, $p\geq 1$, i.e.,
\begin{equation}B_{p}^n(r) := \left\{ (x_1, \ldots, x_n) \in \mathbb{Z}^n: |x_1|^p + \ldots + |x_n|^p \leq r^p \right\}
\end{equation}
and in the $\ell_{\infty}$ metric, ${B_{\infty}^n(r) :=\{ (x_1, \ldots, x_n) \in \mathbb{Z}^n : \max\{|x_1|,\ldots,|x_n|\} \leq r \}.}$

The set  $\mathscr{C} \subset \mathbb{Z}^{n}$ of such a tiling is also called a \textit{perfect code} in the $\ell_p$ metric, $1 \leq p \leq \infty$. If, in addition, this set is an additive subgroup $\Lambda$ of $\mathbb{Z}^n$, we call the corresponding tiling a \textit{lattice tiling}, and the corresponding code a \textit{linear perfect code}. We are interested in characterizing the triples $(n, r, p)$ that admit perfect codes in the $\ell_p$ metric,  $1 \leq p \leq \infty$, with a focus on non-existence results.

For $p = 1$, the existence of such tilings was first investigated by Golomb-Welch \cite{Golomb}. The so-called Golomb and Welch conjecture states that there are no tilings with parameters $(n, r, 1)$, for $n \geq 3$ and $r \geq 2$. Although there have been many advances and partial results since the work of Golomb and Welch, the general conjecture remains open (see \cite{Horak} for further references).

By considering the union of unit cubes in $\mathbb{R}^n$ centered at the points of $B_p^n(r)$ a shape called a \textit{polyomino} is produced. A tiling of $\mathbb{Z}^n$ by translates of $B_p^n(r)$ is a tiling of $\mathbb{R}^n$ by the corresponding polyominoes. We use the notation
\begin{equation} T_p^{n}(r) = \bigcup_{\xx \in B_p^n(r)}{\left( \xx + \left[\frac{-1}{2},\frac{1}{2}\right]^n\right)}, \,\, 1 \leq p \leq \infty,
\label{eq:defPol}
\end{equation}
for the polyomino in the $\ell_p$ metric associated to $B_p^n(r)$. Some polyominoes are depicted in Figure \ref{fig:poliominos}.

\begin{figure}[htb!]
\begin{minipage}[b]{0.30\linewidth}
\centering
\includegraphics[scale=0.3]{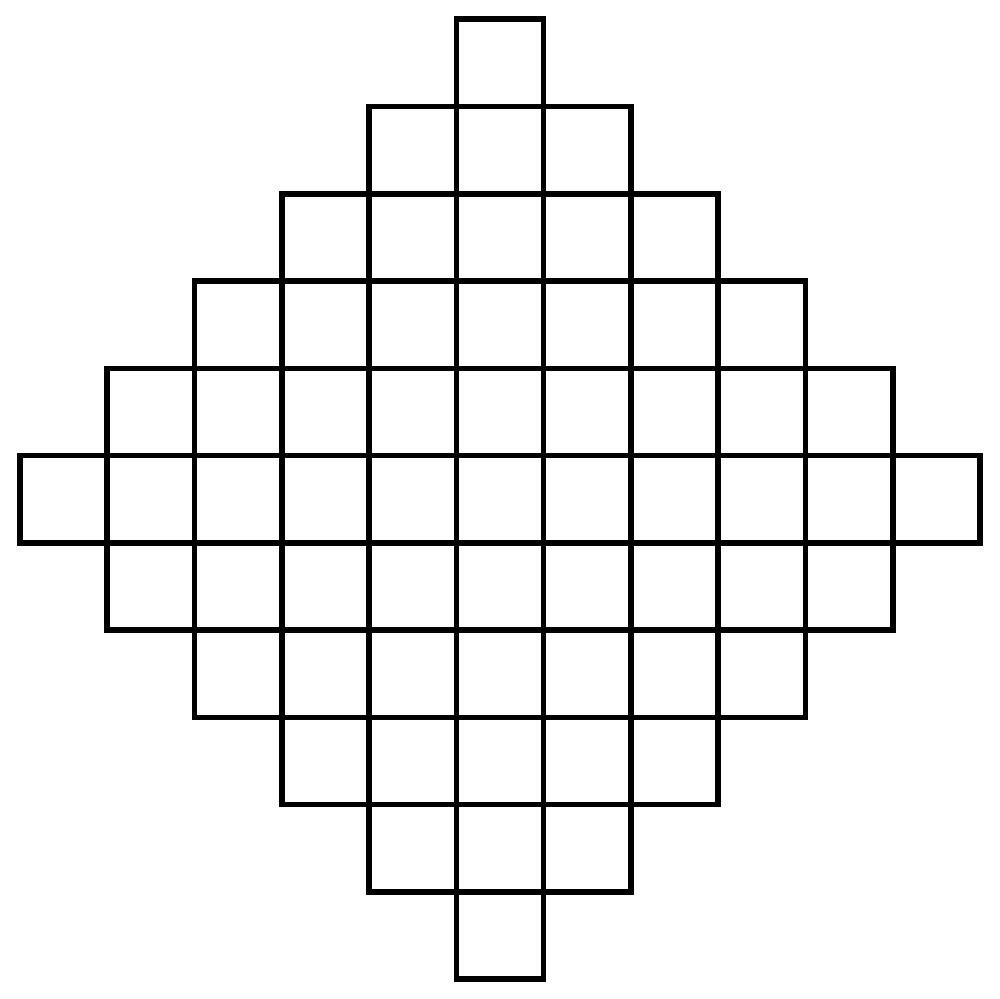}
\end{minipage}
\begin{minipage}[b]{0.30\linewidth}
\centering
\includegraphics[scale=0.3]{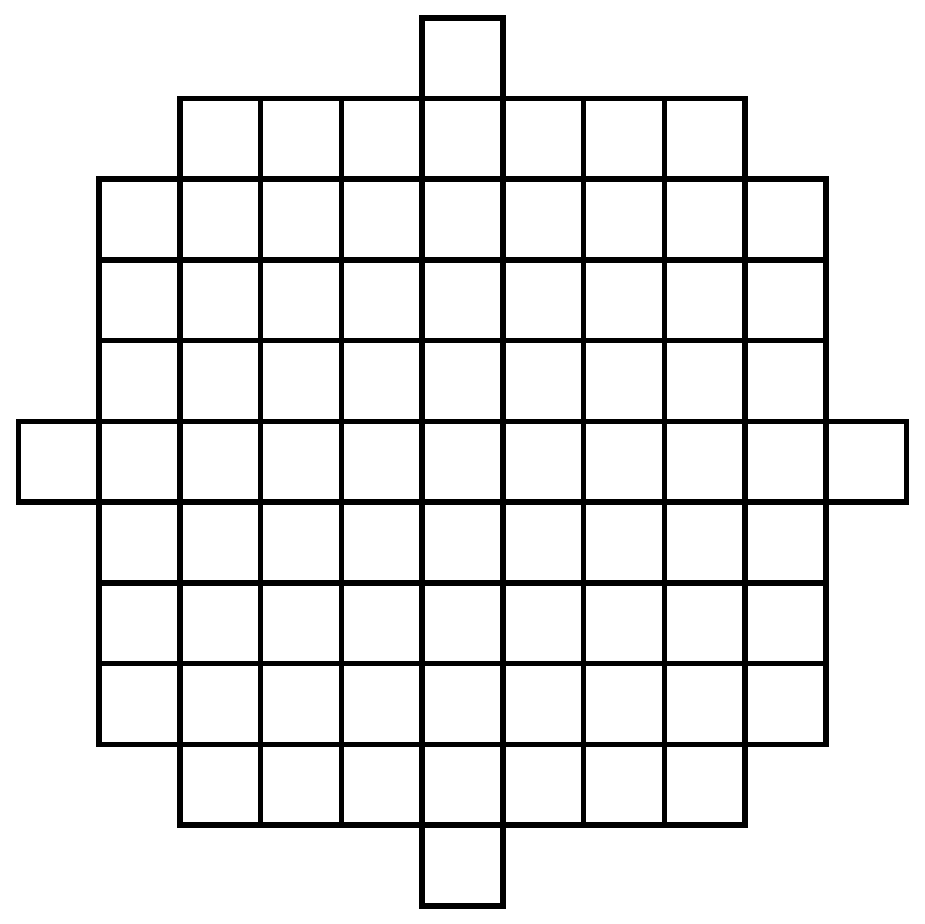}
\end{minipage}
\begin{minipage}[b]{0.30\linewidth}
\centering
\includegraphics[scale=0.3]{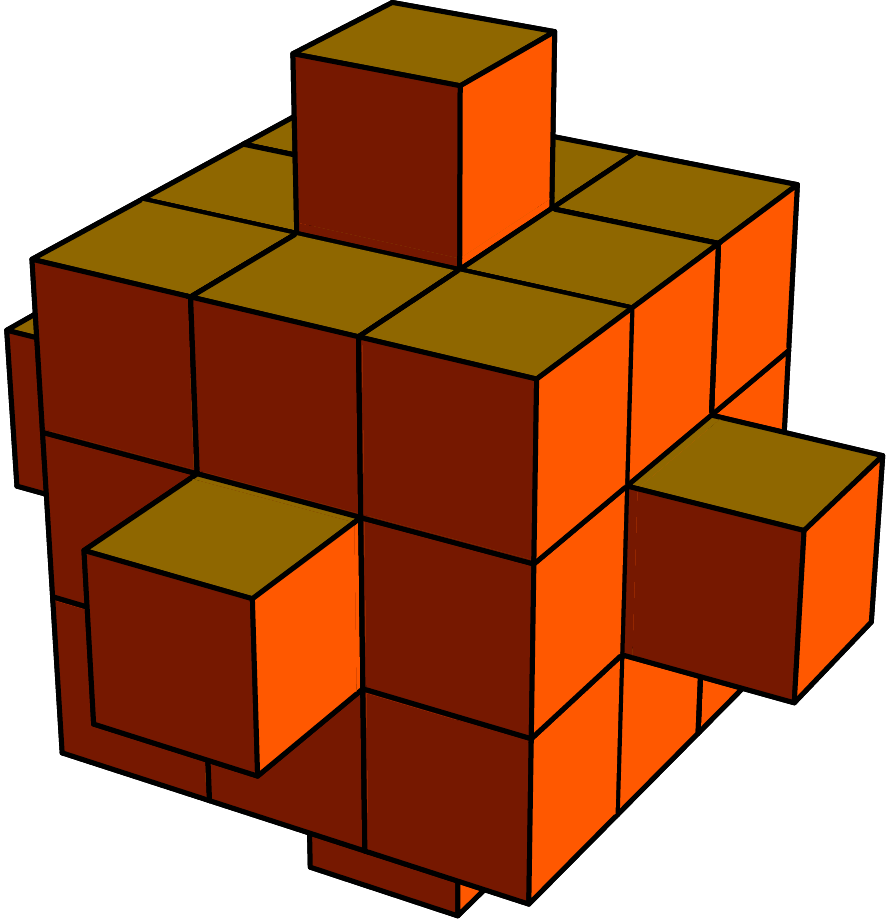}
\end{minipage}
\caption{From the left to the right: The polyominoes $T_1^{2}(5), T_2^{2}(5)$ and $T_4^{3}(2)$.}
\label{fig:poliominos}
\end{figure}

The main results of this paper are described next:

\begin{itemize}
\item[(i)] If a linear perfect code in $\mathbb{Z}^{n}$ (lattice tiling) has parameters $(n,r,p)$, $1\leq p < \infty$, then
\begin{equation}r \leq \frac{n^{1/p}}{2} \frac{\left(1+(\Delta_p^{n})^{1/n}\right)}{\left(1-(\Delta_p^{n})^{1/n}\right)},
\end{equation}
where $\Delta_p^{n}$ is the best packing density of a \textit{superball} \cite{Superballs} (a ball in the $\ell_p$ metric in $\mathbb{R}^n$). This means that for fixed $p$, $r = O(n^{1/p})\,\,$ (Corollary \ref{cor:nonExistenceAsymptotic}).

\item[(ii)] There are no lattice tilings with parameters $(2,r,2)$, except if $r = 1, \sqrt{2}, 2, 2\sqrt{2}$. The are no tilings with parameters $(3, r, 2)$, except if $r = 1, \sqrt{3}$ (Theorem \ref{thm:dimension2} and \ref{thm:dimension3}). 
 
\item[(iii)] For $r$ integer, a $r$-perfect code in the $\ell_{\infty}$ metric is also $n^{1/p}r$-perfect code in the $\ell_p$ if $p>\log(n)/\log(1+1/r)$, which assures the existence of the latter under this condition (Section \ref{sec:infnorm}). The same inequality is required for a tilling by cubic polyominoes in the $\ell_p$ metric (Corollary \ref{cor:perfectinfty}).
 
\item[(iv)] If $r$ is a positive integer, there is no tiling of $\mathbb{R}^{2}$ by translates of $T_{p}^{2}(r)$, $1 \leq p < \infty$, unless $r \leq 2$. Moreover, if $(n-1)(r-1)^{p}+(r-2)^{p} \leq r^{p}$, then there is no tiling of $\mathbb{R}^{n}$ by translates of $T_p^{n}(r)$, unless $r \leq 2$
(Theorem \ref{z2integer} and \ref{thm:tiledimensionn}).
\end{itemize}

Perfect codes in $\mathbb{Z}^n$ in the $\ell_p$ metric are in correspondence with perfect codes over the alphabet $\mathbb{Z}_q$, when the alphabet size $q$ is sufficiently large, in the induced $p$-Lee metric (see Section \ref{sec:perfectCodesAndTilings} for precise definitions). Recently, Sol\'e and Belfiore \cite{BelfioreSole} have employed codes in the Euclidean ($2$-Lee) metric in $\mathbb{Z}_q^n$ as building blocks for asymptotically good spherical codes. They pose the existence of perfect codes in the Euclidean metric as an open question. The results (i) and (ii) above serve as a negative answer when the alphabet size and packing radius are large enough, and for low dimensions ($n = 2, 3$). We point out that the codes used in \cite{BelfioreSole} have astonishingly large alphabet sizes (see Thm. 8 of \cite{BelfioreSole}).

In a broader context, it is fair to say that tilings of $\mathbb{R}^n$ by polyominoes and their applications have been extensively studied in the literature \cite{PPP, Horak, Golomb, Schwartz}. For instance, recent studies of tilings by quasi-crosses \cite{Schwartz2, Schwartz} have applications in coding for flash memories \cite{Schwartz2}. It is worth noting that some of the group-theoretic techniques for proving non-existence of tilings by quasi-crosses in \cite{Schwartz2} are the main tools for our characterization of the $(2,r,2)$ and $(3,r,2)$ tilings in Section \ref{sec:casestudy}. This relation between group homomorphisms and tilings actually dates back to works of Stein and Szabo (see, e.g., \cite{Szabo} for a very complete description of several results in this direction). A recent work \cite{Kovacevic} applies essentially the same techniques to study connection between perfect codes (tilings) in $A_n$ lattices and difference set problems.

\section{Preliminaries}
\subsection{Codes and Lattices}
A $q$-ary \textit{code} $\mathscr{C}$ is a subset of $\mathbb{Z}_{q}^{n}$. A code $\mathscr{C} \subset \mathbb{Z}_q^{n}$ is called \textit{linear} if it is a $\mathbb{Z}_q$-submodule of $\mathbb{Z}_q^n$, $q \in \mathbb{N}$ (i.e., closed under addition). A \textit{lattice} is a $\mathbb{Z}$-submodule of $\mathbb{Z}^n$.  A lattice will be also called a linear code in $\mathbb{Z}^n$. We consider here full rank lattices in $\mathbb{Z}^n$, that is, full rank additive subgroups of $\mathbb{Z}^n$. A lattice $\Lambda$ always has a \textit{generator matrix} $B$, i.e., a full rank matrix such that $\Lambda = \left\{ \bm{x} B : \bm{x} \in \mathbb{Z}^n \right\}$. The {\it determinant} of a lattice is $\det \Lambda = |\det B|,$ which is the Euclidean volume of the parallelotope $\mathcal{P} = \left\{ \bm{\alpha} B : \bm{\alpha} \in [0,1)^n\right\}$, generated by the rows of $B$.

A linear code in $\mathbb{Z}_q^n$ can be extended to a linear code in $\mathbb{Z}^n$ using the so-called Construction A, as follows. Let $\mathscr{C} \subset \mathbb{Z}_q^{n}$ be a code, and consider the map
\begin{equation}
\begin{split}
& \,\,\,\,\,\,\,\,\,\,\,\,\,  \phi:  \mathbb{Z}^n
\longrightarrow \mathbb{Z}_q^n  \\&  (x_1,\ldots,x_n)
\longmapsto (\overline{x_1},\ldots,\overline{x_n}),
 \end{split}
\end{equation}
where $\overline{x_i} = x_i \mbox{ (mod q)}$. Take $\Lambda(\mathscr{C}) = \phi^{-1} (\mathscr{C})$. If $\mathscr{C}$ is linear, then $\Lambda(\mathscr{C})$ is a lattice such that $q \mathbb{Z}^n \subset \Lambda(\mathscr{C})$. The quotient group $\Lambda(\mathscr{C})/q\mathbb{Z}^n$ is isomorphic to $\mathscr{C}$. From this, we have $$\left| \displaystyle\frac{\Lambda
(\mathscr{C})}{q\mathbb{Z}^n} \right| = \displaystyle\frac{q^n}{\det
\Lambda(\mathscr{C}) } = |\mathscr{C}|.$$

\subsection{The $p$-Lee distance}
Codes are often studied in the literature endowed with the Hamming or Lee (Manhattan) metrics. In this work, we will consider an extension of the Lee metric in $\mathbb{Z}_q^n$, induced by the $\ell_p$ metric in $\mathbb{Z}^n$, $1 \leq p \leq \infty$. Recall that the $\ell_p$ distance between two points $\bm{x},\bm{y} \in \mathbb{Z}^n$ is defined as
\begin{equation} \label{dp} d_{p}({\bm x},{\bm y}) =
\left(\displaystyle\sum_{i=1}^n|x_i-y_i|^{p}\right)^{1/p} \mbox{ if } 1 \leq p < \infty \end{equation}
and $d_{\infty}({\bm x},{\bm y}) = \max\left\{|x_i-y_i|; \,\, i=1,\ldots,n\right\}.$
The Lee distance between two elements $\overline{x}, \overline{y} \in \mathbb{Z}_q$ is defined
\begin{equation}
\label{eq:DistanciaDeLee1} d_{Lee}(\overline{x},\overline{y}) =
\displaystyle \min \left\{(\overline{x}-\overline{y}) \,\, ( \mbox{mod
}q), (\overline{y}-\overline{x}) \,\, ( \mbox{mod
}q) \right\}.\end{equation}
For two vectors $\overline{\bm{\xx}},\overline{\bm{\yy}} \in \mathbb{Z}_q^n$, we define the \textit{$p$-Lee distance} as:
\begin{equation} d_{p,Lee}(\overline{\bm x},\overline{\bm y})=
\left(\displaystyle \sum_{i=1}^n
\left(d_{Lee}(\overline{x}_i,\overline{y}_i)\right)^{p}\right)^{1/p} \mbox{ if }  1 \leq p < \infty \label{eq:pLeeDefinition}\end{equation} and \begin{equation} d_{\infty,Lee}(\overline{\bm x},\overline{\bm y})=\max\{d_{Lee}(\overline{x}_i,\overline{y}_i),\,\, i=1,\ldots,n\}.\label{eq:inftyLeeDefinition}\end{equation}
$d_{p,Lee}$ is induced by the $\ell_p$ distance, identifying $\mathbb{Z}_q^n$ with the quotient group $\mathbb{Z}^n/q\mathbb{Z}^n$ and taking the minimum distance between two classes. More analytically, let $\overline{\bm{x}} = \phi(\bm{x})$, and $\overline{\bm{y}} = \phi(\bm{y})$. Define the induced distance as
$$d_{p,\mbox{\small ind}}(\overline{\bm{\xx}},\overline{\bm{\yy}}) = \inf\{d_{p}(\bm{x^{*}},\bm{y^{*}}) \,\, :\,\,\,{\bm x^{*}}={\bm x} + q {\bm t},\,\, \bm{y^{*}}=\bm{y}+q{\bm w}:
\,\,\, {\bm t}, {\bm w} \in \mathbb{Z}^{n}\}.$$
We have the following:

\begin{proposition}For any $\overline{\bm{\xx}}, \overline{\bm{\yy}} \in \mathbb{Z}_q^n$,
$$d_{p,Lee}(\bm{\overline{\xx}},\bm{\overline{\yy}}) = d_{p,\mbox{\small ind}}(\overline{\bm{\xx}},\overline{\bm{\yy}}),\,\, 1\leq p \leq \infty.$$
\end{proposition}
\begin{proof} In this proof we will only  consider the case $1 \leq p < \infty$ since the case $p=\infty$ is simpler. Let ${\overline{\bm x}} = (\overline{x_1},\ldots,\overline{x_n}),\, {\overline{\bm y}}=(\overline{y_1},\ldots,\overline{y_n})$, so that $\overline{\bm{x}} = \phi(\bm{x})$, $\overline{\bm{y}} = \phi(\bm{y})$, and $0\leq x_i,y_i < q$.
\begin{equation*}\begin{split}d_{p,\mbox{\small ind}}({\overline{\bm
x}},{\overline{\bm y}}) &= \inf\{d_{p}(\bm{x^{*}},\bm{y^{*}}):\,\,\,
{\bm x^{*}}={\bm x} + q {\bm t},\,\, \bm{y^{*}}=\bm{y}+q{\bm w}:
\,\,\, {\bm t}, {\bm w} \in \mathbb{Z}^{n}\} \\& =
\inf\left\{\left(\sum_{i=1}^{n}|x_i-y_i-q(w_i-t_i)|^{p}\right)^{\frac{1}{p}}:
\,\,\, {\bm t}, {\bm w} \in \mathbb{Z}^{n} \right\}\\&
=\left(\sum_{i=1}^{n}\left|x_i-y_i-q\left\lceil
\frac{x_i-y_i}{q}\right\rfloor\right|^{p}\right)^{\frac{1}{p}},
\end{split}\end{equation*} where the last equality follows from the fact that  $
\displaystyle\sum_{i=1}^{n}|x_i - y_i - q s_i|^{p}$ is minimum when $s_i
=\left\lceil \frac{x_i - y_i}{q}\right \rfloor$ (with ties broken to the integer with smallest absolute value) is the closest integer to $\frac{x_i - y_i}{q}$, since the summation
terms are independent.

Let $\alpha_i =  \left\lceil\frac{{x_i}-y_i}{q} \right\rfloor$ for
$i=1,\ldots,n$. For any summation term, since $0\leq |{x_i}-y_i| < q$, it follows that
$-1 < \frac{{x_i}-y_i}{q} < 1$ and $\alpha_i \in \{-1,0,1\}.$
\begin{itemize}
 \item If $\alpha_i = 0$ for some
$i,$ then $-q/2 \leq {x_i}-y_i \leq q/2$ and this implies
$\min\{|x_i-y_i|,q-|x_i-y_i|\} = |x_i-y_i|.$  \item If $\alpha_i =
1$ for some $i,$ then $q/2 < {x_i}-y_i < q $ and then
$\min\{|x_i-y_i|,q-|x_i-y_i|\} = q -|x_i-y_i| \mbox { and }
|x_i-y_i| = x_i-y_i.$ \item If $\alpha_i = -1$ for some $i$, then
$-q < {x_i}-y_i < -q/2$ and then $\min\{|x_i-y_i|,q-|x_i-y_i|\} =
q - |x_i-y_i| \mbox{ and } |x_i-y_i| = -(x_i-y_i).$
\end{itemize} Therefore we can assert $\left|x_i-y_i-q \alpha_i\right|=d_{Lee}(\overline{x_i},\overline{y_i})$, for all $i=1,\ldots,n$, and then
$d_{p,\mbox{\small ind}}({\overline{\bm
x}},{\overline{\bm y}})$ $ = \left(\displaystyle \sum_{i=1}^n
\left(d_{Lee}(\overline{x}_i,\overline{y}_i)\right)^{p}\right)^{1/p}.$
\end{proof}


In what follows we will denote also $d_{p,Lee}$ as $d_p$ and $B_p^n(\bm{x},r)$ will be used for the closed ball either in $\mathbb{Z}^n$ or in $\mathbb{Z}_q^n$ centered at $\bm{x}$ with radius $r$.

The \textit{minimum distance} $d_p(\mathscr{C})$ of a code $\mathscr{C}$ in $\mathbb{Z}^{n}$ or $\mathbb{Z}_q^{n}$ is defined as $$d_{p}(\mathscr{C}) = \min_{{{\bm{\xx}},{\bm{\yy}} \in \mathscr{C} \above 0pt \xx\neq\yy}} d_{p}({\bm{\xx}},{\bm{\yy}}).$$ The minimum distance of a lattice $\Lambda$, $d_p(\Lambda)$, is the $p$-norm of a shortest nonzero vector.

\section{Perfect Codes and Tilings}
\label{sec:perfectCodesAndTilings}
A code $\mathscr{C}$ in  $\mathbb{Z}_q^n$ ($\mathbb{Z}^{n}$) is called \textit{perfect} if, for some $r > 0$, it satisfies the property that, for any $\bm{z} \in \mathbb{Z}_q^n$ ($\mathbb{Z}^{n}$), there exists only one ${\bm{x}} \in \mathscr{C}$ such that $d_{p}({\bm{x}},{\bm{z}}) \leq r$. In other words, the balls centered at codewords are disjoint and cover $\mathbb{Z}_q^n$ ($\mathbb{Z}^{n}$).


\subsection{The Packing Radius}

In order to study perfect codes, we need the concept of packing radius of a code. In the classic Lee metric, the packing radius $r_1=r_1(\mathcal{\mathscr{C}})$ of a code $\mathscr{C}$ is defined as the largest integer $r$ such that the balls of radius $r$ centered at codewords are disjoint. This is a natural definition, since the distance between two points in such metric is always an integer. The extension of this definition to general $p$-metrics requires further discussion.

First note that for $1 \leq p <\infty$, $p \in \mathbb{N}$, the $p$-Lee distance between two vectors in $\mathbb{Z}_q^n$, as well as the $\ell_p$ distance in $\mathbb{Z}^n$, is always the $p$-th root of an integer. Hence, for any two vectors $\xx, \yy$, we have $d_{p}(\xx,\yy)^p \in \mathbb{N}$. However not all $p$-th roots are achievable, as shown in the next example.

\begin{example} Let $p = 2$, $n = 2$. There is no pair of vectors $\overline{\xx},\overline{\yy} \in \mathbb{Z}_q^2$ such that $d_{2}(\overline{\xx},\overline{\yy}) = \sqrt{3}$. Hence $B_{2}^{2}(\sqrt{2}) = B_{2}^{2}(\sqrt{3})$.
\end{example}

Therefore, in order to define the packing radius, we need to look for the largest value \textit{within the set of achievable distances} such that the balls centered at codewords are disjoint. We thus define the \textit{distance set} of the $p$-Lee metric in $\mathbb{Z}_q^n$ as the set $\mathcal{D}_{p,n,q}$ of all achievable distances. Analogously, we define the set $\mathcal{D}_{p,n}$ of achievable distances by the $\ell_p$ metric in $\mathbb{Z}^n$. It follows that $\mathcal{D}_{p,n} \subset \left\{0, 1^{1/p}, 2^{1/p}, 3^{1/p},\ldots \right\},$ $1 \leq p < \infty,$ and $\mathcal{D}_{\infty,n} = \mathbb{N}$. The following number theoretic considerations concerning the distance set are useful:

\begin{proposition}\label{thm:distanceSet} Let $\mathcal{D}_{2,n}$ be a distance set as defined above. A number ${r \in \left\{0, 1^{1/2}, 2^{1/2}, 3^{1/2},\ldots \right\}}$ is in $\mathcal{D}_{2,n}$ if and only if
\begin{enumerate}
\item $n = 2$ and $r^2$ can be written as $r^2 = ab^2$, where $a$ has no prime factor congruent to $3$, modulo $4$, or
\item $n = 3$ and $r^2$ is not of the form $4^m(8k+7)$ for $m, k \in \mathbb{N}$, or
\item $n \geq 4$.
\end{enumerate}
\end{proposition}
\begin{proof} The proof follows by observing that $r \in \mathcal{D}_{2,n}$  if and only if can be written as the sum of $n$ squares, and further applying the corresponding sum of squares theorem (see, e.g., \cite{Niven}).
\end{proof}
There is also a natural connection between $\mathcal{D}_{p,n}$ and Waring's problem (e.g. \cite{SurveyWaring}). Waring's problem asks for the smallest integer $k$ such that all $m \in \mathbb{N}$ can be written as $m = x_1^p + \ldots + x_k^p$, where $x_i \in \mathbb{N}$. This problem is well defined, in the sense that such $k$ always exists. Let $g(p)$ be the solution to Waring's problem, given $p > 1$. It follows directly that, if $n \geq g(p)$, then $\mathcal{D}_{p,n} = \left\{0, 1^{1/p},2^{1/p},\ldots \right\}$ i.e., all $p$-th roots of integers are achievable. In the survey \cite{SurveyWaring}, bounds on $g(p)$ are presented. In particular, it is shown that $g(2) = 4, g(3) = 9, g(14) = 19$. It is conjectured (and proved for $p \leq 471,600,000$) that $g(p) = 2^{p} + \left\lfloor (3/2)^p \right\rfloor -2$.

The \textit{packing radius} of a code $\mathscr{C} \subset \mathbb{Z}_q^{n}$ is thus defined as the largest $r \in \mathcal{D}_{p,n,q}$ such that $B_{p}^n(\overline{\bm{\xx}},r) \cap B_{p}^n(\overline{\bm{\yy}},r) = \emptyset$ holds for all $\overline{\xx}, \overline{\yy} \in \mathscr{C}$. Analogously, the packing radius\footnote{In the literature, the term \textit{packing radius} is usually related to the Euclidean packing radius of a lattice  in $\mathbb{R}^n$ (i.e., the maximum radius such that open Euclidean balls centered at lattice points are disjoint). We say more about this in Section 5.} of a code $\mathscr{C} \subset \mathbb{Z}^{n}$ is the largest $r \in \mathcal{D}_{p,n}$ such that $B_{p}^n(\bm{\xx},r) \cap B_{p}^n(\bm{\yy},r) = \emptyset$ holds for all $\xx, \yy \in \mathscr{C}$. The packing radius of a code $\mathscr{C}$ (in $\mathbb{Z}^{n}$ or $\mathbb{Z}_q^{n}$) in the $\ell_p$ metric will be denoted by $r_p(\mathscr{C})$.

For $p = 1$, it is well known that the packing radius of a code $\mathscr{C}$ (in $\mathbb{Z}^n$ or in $\mathbb{Z}_q^n$ ) is given by the formula
$r_1(\mathscr{C}) = \left\lfloor \frac{d_{1}(\mathscr{C})-1}{2} \right\rfloor.$ For $p=\infty$, it is also the case that $r_{\infty}(\mathscr{C}) =  \left\lfloor \frac{d_{\infty}(\mathscr{C})-1}{2} \right\rfloor.$ However, this is not true for $1 < p < \infty$. In fact, as will be shown later, two codes with same minimum distance may have different packing radii, thus the packing radius is not uniquely determined by the minimum distance.

A perfect code $\mathscr{C}$ with packing radius $r = r_p(\mathscr{C})$ is also denominated an \textit{$r$-perfect code}.
%
%
\subsection{Large Alphabet versus Small Alphabet}

Due to applications in coding theory, the set $\mathbb{Z}_q$ is also denominated an alphabet (and $q$ is the alphabet size). There is a big difference between large alphabets and small alphabets in terms of the search for perfect codes. If the alphabet size $q$ is large enough, then perfect linear codes in $\mathbb{Z}_q^n$ induce perfect linear codes in $\mathbb{Z}^n$ through Construction A. On the other hand, if $q$ is small this need not be the case, as shown in the next example.

\begin{example} The binary code $\mathcal{\mathscr{C}}=\{(\overline{0},\overline{0},\overline{0},\overline{0},\overline{0},\overline{0},\overline{0}),(\overline{1},
\overline{1},\overline{1},\overline{1},\overline{1},\overline{1},\overline{1})\}
\subset \mathbb{Z}_2^{7}$ is perfect in the Lee (Hamming) metric, but the associated  lattice $\Lambda(\mathcal{\mathscr{C}})$ is not a perfect code in $\mathbb{Z}^{7}$.
\end{example}

For large enough alphabet size ($q \geq d_{p}(\mathscr{C})$), a code and its associated lattice  $\Lambda(\mathcal{\mathscr{C}})$ have same distance \cite{Superballs}:
\begin{equation}
d_p(\Lambda(\mathscr{C})) = \min\left\{d_{p}(\mathscr{C}),q\right\}.
\end{equation}
A related result is the following:

\begin{proposition} Let $\mathscr{C} \subset \mathbb{Z}_q^n$ be a linear code with packing radius $r_p=r_p(\mathscr{C})$. If $2 r_p < q$, then $\Lambda(\mathscr{C}) \subset \mathbb{Z}^n$ is a code with packing radius $r_p$ in the $\ell_p$ metric.
\label{thm:samePackingRadius}
\end{proposition}
\begin{proof}
Let $r_p(\Lambda(\mathscr{C}))$ be the packing radius of $\Lambda(\mathscr{C})$. It is clear that $r_p(\Lambda(\mathscr{C}))$ cannot exceed $r_p(\mathscr{C})$. We prove that $r_p(\Lambda(\mathscr{C}))$ is indeed equal $r_p(\mathscr{C})$. Suppose that there exists $\xx \in \mathbb{Z}^n$ and two elements $\bm{c},\bm{c}' \in \Lambda(\mathscr{C})$ such that $d_p(\xx,\bm{c}) \leq r_p(\mathscr{C})$ and $d_p(\xx, \bm{c}') \leq r_p(\mathscr{C})$. By reducing the three vectors modulo $q$, we have $d_p(\overline{\bm{c}},\overline{\xx}) \leq r_p(\mathscr{C})$ and $d_p(\overline{\bm{c}}',\overline{\xx}) \leq r_p(\mathscr{C})$, with $\overline{\bm{c}}, \overline{\bm{c}}' \in {\mathscr{C}}$. This implies that $\overline{\bm{c}} = \overline{\bm{c}}'$, i.e., $\bm{c} = \bm{c}' + q \bm{w}$, $\bm{w} \in \mathbb{Z}^n$ and $d_p(\bm{c}, \bm{c}') = q d_p(\bm{w},\bm{0})$. But
$$d_p(\bm{c}, \bm{c}') \leq d_p(\bm{c}, \xx) + d_p(\xx, \bm{c}') \leq 2r_p(\mathscr{C}) < q,$$
therefore $\bm{w} = \bm{0}$ and $\bm{c} = \bm{c}'$. This proves that the balls of radius $r_p(\mathscr{C})$ centered at codewords are disjoint. Therefore $r_p(\Lambda(\mathscr{C})) \geq r(\mathscr{C})$, concluding the proof.
\end{proof}

\begin{corollary}
\label{cor:samePackingRadius} If ${\mathscr{C}} \subset \mathbb{Z}_q^n$ is a linear perfect code with packing radius $r_p = r_p(\mathscr{C})$ such that  $2r_p < q$, then $\Lambda(\mathscr{C}) \subset \mathbb{Z}^n$ is a $r_p$-perfect code in the $\ell_p$ metric.
\end{corollary}

\section{On Perfect Codes in the $\ell_\infty$ and $\ell_p$ metric}
\label{sec:infnorm}
The polyominoes $T_{\infty}^{n}(r)$ (recall Eq. \ref{eq:defPol}) are cubes, thus clearly $T_{\infty}^{n}(r)$ tiles $\mathbb{R}^n$ for any $r \in \mathbb{N}$. Over finite alphabets, we have the following characterization:

\begin{proposition} \label{necessary} There are non-trivial perfect codes $\mathscr{C} \subset \mathbb{Z}_q^n$ in the $\ell_{\infty}$ metric iff $q = b m$ with $b > 1$ an odd integer and $m > 1$ an integer.
\end{proposition}

\begin{proof} A code $\mathscr{C}
\subset \mathbb{Z}_q^{n}$ with minimum distance $d_{\infty}(\mathscr{C}) = 2r+1$ is perfect
with respect to the distance $d_{\infty}$ iff
\begin{equation}\label{condition}  |\mathscr{C}| (2r+1)^{n} = q^{n}. \end{equation}

\begin{itemize} \item Necessary condition: From \eqref{condition}, if there exists a perfect code $\mathscr{C}$, then
$$|\mathscr{C}| = \left(\frac{q}{2r+1}\right)^{n}.$$
Hence $2r+1$ has to divide $q$. Excluding the trivial codes ($r = 0$), $q$
must have an odd factor greater than $1$. If $q$ is prime, since $(2r+1) | q$ it follows that
$2r+1 = q$, which gives a trivial perfect code.

\item Sufficient condition: Let $q=b m$ with $b
> 1$ an odd integer and $m > 1$ an integer. Taking the code $\mathscr{C}$
generated by the vectors
$$\{(\overline{b},\overline{0},\ldots,\overline{0}),(\overline{0},\overline{b},\ldots,\overline{0})\ldots,(\overline{0},\ldots,\overline{0},\overline{b})\}
\subset \mathbb{Z}_q^{n}$$ we have $|\mathscr{C}| = m^n$. In fact, if
$\overline{t} \in \mathbb{Z}_q$ and $\overline{t} =
\overline{a}\overline{m} + \overline{r}$ with $\overline{0}\leq
\overline{r} <\overline{m}$ then
$\overline{t}(\overline{0},\ldots,\overline{b},\ldots,\overline{0})=\overline{r}(\overline{0},\ldots,\overline{b},\ldots,\overline{0})$.
For this code the minimum distance $\mu =
\min\{d_{\infty}(\overline{x},\overline{y}); \,\,
\overline{x},\overline{y} \in \mathscr{C}, \,\, \overline{x}\neq \overline{y}
\}=b$. Therefore $r = (b-1)/2$. Since $(2r +
1)^{n} = b^{n}$, $\mathscr{C}$ is a $r$-perfect code.
\end{itemize}\end{proof}
In the last proposition we have obtained perfect codes in the $\ell_{\infty}$
metric that are Cartesian products. The next example shows that
there are other families of perfect codes in the $\ell_{\infty}$ metric.

\begin{example} Let $q=m^{2}$, where $m$ is an odd prime number. The group $\mathbb{Z}_{m^{2}}^{2}$ has $m^{4}$ elements
and has non-trivial subgroups of order $m$, $m^{2}$ and $m^{3}$. From the equality $|\mathscr{C}|(2r+1)^{2}=m^{4}$, a necessary condition to obtain a non-trivial  perfect code in the $\ell_{\infty}$ metric is $|\mathscr{C}|=m^{2}$. The codes
$\mathscr{C}=\left<(\overline{1},\overline{m\,a})\right>$, $\overline{a} \neq
\overline{0}$, satisfy $|\mathscr{C}|=m^{2}$ and they are perfect. Figure \ref{fig:noncyclic}
shows the code $\mathscr{C}=\left<(\overline{1},\overline{7})\right> \subset
\mathbb{Z}_{49}^{2}$.
\begin{figure}[h!]
\centering
\includegraphics[scale=0.4]{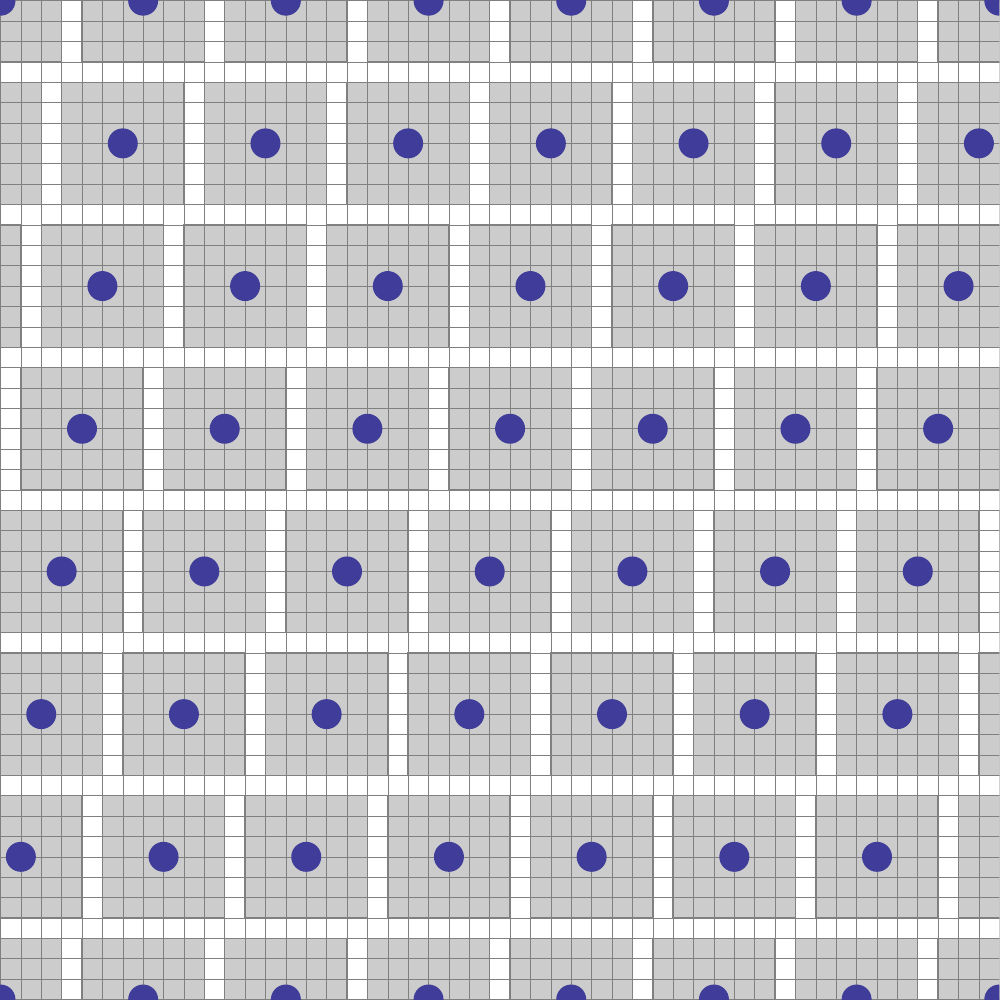}
\caption{The code $\mathscr{C}=\left<(\overline{1},\overline{7})\right>
\subset \mathbb{Z}_{49}^{2}$, which is 3-perfect in the $\ell_{\infty}$ metric, and also $2^{1/p} 3$-perfect in the $\ell_p$ metric, for $p \geq 3$ (Corollary \ref{cor:perfectinfty}).}
\label{fig:noncyclic}
\end{figure}
\label{ex:noncyclic}
\end{example}

The following proposition relates polyominoes in the $\ell_p$ and $\ell_\infty$ metrics.

\begin{proposition} \label{p-larger} If $n r^p < (r+1)^p,$ $r$ integer, then $B_p^{n}(n^{1/p}r) = B_{\infty}^{n}(r)$. On the other hand, if $B_p^{n}(r_1) = B_\infty^{n}(r_2)$, $r_2$ integer, for some $r_1, r_2 \geq 1$, then $n r_2^p \leq r_1^p < (r_2+1)^p$. 
\label{prop:cubicPolyominoes}
\end{proposition}
\begin{proof} The proof follows from the fact that for any $\xx$ and $\yy \in \mathbb{Z}^n$, we have
\begin{equation}
d_{\infty}(\xx,\yy) \leq d_{p}(\xx,\yy) \leq n^{1/p} d_{\infty}(\xx,\yy).
\end{equation}
For the first part, the inequalities imply that $ B_\infty^n(r) \subset B_p^n(n^{1/p} r)$ and that any point $\xx \in B_p^n(n^{1/p} r)$ must satisfy $d_{\infty}(\xx,\bm{0}) < r+1$, thus $B_p^{n}(n^{1/p}r) = B_{\infty}^{n}(r)$.

To prove the ``converse'', suppose $B_p^{n}(r_1) = B_\infty^{n}(r_2)$. Since $(r_2, \ldots, r_2) \in B_{\infty}^{n}(r_2) \Rightarrow$
\begin{equation} (r_2,\ldots,r_2) \in B_{p}^{n}(r_1) \Rightarrow n r_2^p \leq r_1^p.
\label{eq:simple}
\end{equation}
On the other hand, since $(r_2+1, 0 \ldots, 0) \notin B_\infty^{n}(r_2)= B_p^{n}(r_1)$, we conclude that $(r_2+1)^p > r_1^p$.
\end{proof}

\begin{corollary} If $p> \log(n)/\log(1+1/r)$, $r$ integer, an $r$-perfect code $\mathscr{C} \subset \mathbb{Z}^n$ in the $\ell_{\infty}$ metric is also $n^{1/p}r$-perfect  in the  $\ell_{p}$ metric, which assures the existence of perfect codes in the $\ell_p$ metric under this condition. On the other hand, if an $h$-perfect code $\mathscr{C} \subset \mathbb{Z}^{n}$ in the $\ell_p$ metric is also an $r$-perfect in the $\ell_{\infty}$ metric (cubic tiles $T_p^{n}(h)$), then $h=n^{1/p}r$ and $p>\log(n)/\log(1+1/r).$
\label{cor:perfectinfty}
\end{corollary}

\begin{example} The Example \ref{ex:noncyclic} can be extended by considering the $\ell_p$ metric. The code $\mathscr{C}=\left<(\overline{1},\overline{m\,a})\right>$, $\overline{a} \neq
\overline{0}$, $m$ prime, is perfect in the $\ell_p$ metric with packing radius $r_p = 2^{1/p}(m-1)/2$, if $(1 + 2/(m-1))^p > 2$. In particular, the code $\mathscr{C}=\left<(\overline{1},\overline{7})\right> \subset
\mathbb{Z}_{49}^{2}$ in Figure \ref{fig:noncyclic} is also perfect in the $\ell_p$ metric for $p \geq 3$.
\end{example}

For fixed $p$ and $n$, the conditions of Corollary \ref{cor:perfectinfty} can be quite strict. For $p = 2$, the condition of Corollary \ref{cor:perfectinfty} is only satisfied in $\mathbb{Z}^2$ ($r = \sqrt{2}, 2\sqrt{2}$) and in $\mathbb{Z}^3$ ($r = \sqrt{3}$).

We have seen above that perfect codes in the $\ell_{\infty}$ metric are also perfect codes in the $\ell_p$ metric, for large $p$. But we also may have codes which are perfect in the $\ell_p$ metric for any $p \geq 1$ but not perfect in the $\ell_{\infty}$ metric.

\begin{example} From Proposition \ref{necessary}, if $q$ is a prime number, then there are no perfect codes $\mathscr{C} \subset \mathbb{Z}_q^{n}$
considering the $\ell_{\infty}$ metric. However, the $13$-ary code $\mathscr{C} =
\langle (\overline{1},\overline{5})\rangle \subset
\mathbb{Z}_{13}^{2}$ is perfect in the $\ell_{p}$ metric for $1 \leq p
<\infty$ (Figure \ref{fig:perfectl2}).
\label{ex:bla}
\end{example}

\begin{remark} Golomb and Welch \cite[Thm. 3]{Golomb} showed that there are $1$-perfect codes in the $\ell_1$ metric for all $n$. Since $B_p^{n}(1) = B_1^{n}(1)$, $1 \leq p < \infty$, this also proves the existence of $1$-perfect codes in the $\ell_p$ metric for all $n$.
\label{rmk:1isTrivial}
\end{remark}

%

\section{Non-Existence Theorems}

\subsection{Asymptotic Non-Existence Theorems}
\label{subsec:Asymptotic}
We use the technique of \cite{Golomb} dealing with the Lee metric ($p=1$), recently revisited by Horak \cite{Horak}, to prove the non-existence of perfect codes in the $\ell_p$ metric, $1 < p < \infty$, for sufficiently large $r$. The idea is to associate a packing by \textit{superballs} to a tiling of polyominoes of approximately the same radius. A first difficulty for extending the results for $p>1$ is that the packing radius of a code is not completely determined by its minimum distance, as shown next.

\begin{remark} Let $1 < p < \infty$, $p \in \mathbb{N}$. For $n \geq 2^p$, there exist two lattices $\Lambda_1, \Lambda_2 \subset \mathbb{Z}^n$ such that $d_p(\Lambda_1) = d_p(\Lambda_2)$ and $r_p(\Lambda_1) \neq r_p(\Lambda_2)$.  For instance, let $q = 4$. Consider two codes, ${\mathscr{C}}_1 = \left\{ (\overline{2},\overline{0},\ldots,\overline{0})j : j = \overline{0},\overline{1} \right\} \subset \mathbb{Z}_4^n$, and $${\mathscr{C}}_2 = \{ (\overbrace{\overline{1},\overline{1},\ldots,\overline{1}}^{2^p },\overline{0},\ldots,\overline{0})j : j = \overline{0}, \overline{1}, \overline{2}, \overline{3} \} \subset \mathbb{Z}_4^n.$$
It is clear that $d_{p}({\mathscr{C}}_1) = d_{p}({\mathscr{C}}_2) = 2$. On the other hand $r_p({\mathscr{C}}_1) = 0$, since $(\overline{1},\overline{0},\ldots,\overline{0})$ is at distance $1$ from the origin and from $(\overline{2},\overline{0},\ldots,\overline{0})$, whereas $r_p({\mathscr{C}}_2) = (2^{p-1}-1)^{1/p}$. In other words, ${\mathscr{C}}_1$ and ${\mathscr{C}}_2$ have same $\ell_{p}$ distance but different radii. From Proposition \ref{thm:samePackingRadius}, this property is transferred to the lattices $\Lambda({\mathscr{C}}_1)$ and $\Lambda({\mathscr{C}}_2)$.
\end{remark}

However, it is possible to lower/upper bound $r_p(\Lambda)$ in terms of $d_p(\Lambda)$, so that, rouhgly, for large packing radius, $r_p(\Lambda) = d_p(\Lambda)/2+O(1)$. The following results provide such bounds.

\begin{lemma}The minimum distance and packing radius of a lattice $\Lambda \subset \mathbb{Z}^n$ satisfy
\begin{equation*} \left\lfloor \frac{d_p(\Lambda)-1}{2} \right\rfloor \leq r_p(\Lambda) <  \frac{d_p(\Lambda)}{2} + \frac{n^{1/p}}{2}.
\end{equation*}
\label{thm:radiusAsymptotic}
\end{lemma}
\begin{proof}
\begin{itemize}
\item[(i)] Lower bound: It is enough to observe that $ \left\lfloor \frac{d_p(\Lambda)-1}{2} \right\rfloor$ belongs to $\mathcal{D}_{p,n}$ and apply standard arguments.
\item[(ii)] Upper Bound: Let $r_u := \frac{d_p(\Lambda)}{2} + \frac{n^{1/p}}{2}$. Let $\bm{x} \in \Lambda$ be a vector such that $d_{p}(\bm{x},\bm{0}) = d_p(\Lambda)$. Suppose wlog that $x_i \geq 0$ for all $i$. We prove that $B_p(\bm{0}, r_u) \cap B_p(\bm{x}, r_u) \neq \emptyset$. For this let
$$\yy = \left(\left\lfloor \frac{x_1}{2}\right\rfloor,\left\lfloor \frac{x_2}{2}\right\rfloor,\ldots, \left\lfloor \frac{x_n}{2}\right\rfloor\right).$$
Of course $d_p(\yy,\bm{0}) \leq d_p(\Lambda)/2$. On the other hand,
$$d_p(\xx,\yy) = \left(\sum_{i=1}^n \left|x_i - \left\lfloor \frac{x_i}{2} \right\rfloor \right|^p\right)^{1/p} \leq \left(\sum_{i=1}^n \left|\frac{x_i}{2} + \frac{1}{2} \right|^p\right)^{1/p} \stackrel{(a)}{\leq} \frac{d_p(\Lambda)}{2} + \frac{n^{1/p}}{2},$$
where in $(a)$ we used the triangle inequality for the $\ell_p$ metric.
\end{itemize}
\end{proof}
%
Another difficulty is the fact that, unlike the $\ell_1$ metric, there is no closed form for $\mu_p(n,r)$ the number of integer points in $B_p^{n}(r)$, when $1 < p < \infty$. However, the following asymptotic result is enough to our purposes. Let $S_{p}^{n}(r)$ be a $p$-ball (superball) in $\mathbb{R}^n$, i.e.,
$$S_{p}^n(r) = \left\{ \xx \in \mathbb{R}^n: |x_1|^p + \ldots |x_n|^p \leq r^p \right\}.$$

Let $V_{n,p}$ be the (Euclidean) volume of a $p$-ball of radius $1$ in $\mathbb{R}^n$ \cite[p.32]{Volume}.

\begin{lemma} $\displaystyle \lim_{r \to \infty} \frac{V_{n,p} r^n}{ \mu_p(n,r)} = 1.$
\label{prop:detNumberPoints}
\end{lemma}
\begin{proof}Let $\mathcal{A} = [-1/2,1/2]^n$ be a fundamental cube and $l = n^{1/p}/2$ be the maximum $\ell_p$ norm of a point in $\mathcal{A}$. For $r > l$, we have:
$$S_p^n(r-l) \subset \bigcup_{\bm{x} \in B_{p}^n(r)} (\bm{x} + \mathcal{A}) \subset S_p^n(r+l) \Rightarrow$$
$$(r-l)^n V_{n,p} \leq \mbox{vol}\left(\bigcup_{\bm{x} \in B_p^n(r)} (\bm{x} + \mathcal{A})\right) \leq (r+l)^n V_{n,p} \Rightarrow $$
$$(r-l)^n V_{n,p} \leq \mu_{p}(n,r)  \leq (r+l)^n V_{n,p}, $$
and the result follows by dividing the three terms by $r^n V_{n,p}$ and taking the limit as $r \to \infty$.
\end{proof}

Let $\Lambda \subset \mathbb{R}^{n}$ be a lattice with minimum distance  $d_p=d_p(\Lambda)$ in the $\ell_p$ metric. A \textit{superball packing} \cite{Superballs} is the union of the translates of the $p$-ball $S_p^n(d/2)$ by all points of $\Lambda$. To this packing, we associate a \textit{packing density} 
$$\Delta_p^n(\Lambda) = \frac{V_{n,p}(d_p/2)^n}{{\det\Lambda}}.$$
We note that $\Delta_p^{n}(\Lambda) < 1$ for $1 < p < \infty$. Let $\Delta_p^{n} = \sup_{\Lambda} \Delta_p^n(\Lambda) < 1$ be the supremum of the densities over all $n$-dimensional lattices. The following theorem generalizes a result by Golomb-Welch for $p = 1$ (\cite[Thm. 6]{Golomb}).

\begin{theorem} Let $1 < p < \infty$ and $n\geq 2$. There is a radius $\overline{r}_{n,p}$ such that if $r \geq \overline{r}_{n,p}$ no linear $r$-perfect code in $\mathbb{Z}^n$ exists in the $\ell_p$ metric.
\label{thm:nonExistenceAsymptotic}
\end{theorem}
\begin{proof}A linear perfect code $\Lambda$ with radius $r_p = r_p(\Lambda)$ and minimum distance $d_p = d_p(\Lambda)$ induces a superball packing with density $\Delta_p^n(\Lambda) = V_{n,p}(d_p/2)^n/\mu_p(n,r_p)$. From Lemma \ref{prop:detNumberPoints}  $\displaystyle \lim_{r_p \to \infty} \Delta_p^n(\Lambda) = 1$. This means that if there were perfect codes for $r_p$ arbitrarily large, the density of the induced packing would be greater than the supremum of all densities, which is a contradiction. Hence, there is a threshold $\overline{r}_{n,p}$ (depending on $n$ and $p$) such that $r \geq \overline{r}_{n,p}$ implies non-existence of perfect codes.
\end{proof}

The value $\Delta_2^n$ is known for $1 \leq n \leq 8$ and $n = 24$ (see \cite{Cohen} also for bounds in other dimensions). For other values of $p$, not much is known on $\Delta_p^n$.

\begin{corollary}\label{cor:nonExistenceAsymptotic} Let $1 < p < \infty$ and $n \geq 2$. The radius of a linear $r_p$-perfect code in the $\ell_p$ metric satisfies
\begin{equation}  r_p \leq \frac{n^{1/p}}{2} \frac{\left(1+(\Delta_p^{n})^{1/n}\right)}{\left(1-(\Delta_p^{n})^{1/n}\right)}.
\end{equation} 
Asymptotically, as $n \to \infty$, all $r_p$-perfect codes satisfy $r_p = O(n^{1/p})$.
\end{corollary}

\begin{proof} As in Theorem \ref{thm:nonExistenceAsymptotic}, a perfect code $\Lambda$ induces a superball packing with density
\begin{equation}\label{eq:limitRadius} \Delta_p^n(\Lambda) = \frac{V_{n,p}(d_p/2)^n}{\mu_p(n,r_p)} \stackrel{(a)}{\geq} \left(\frac{r_p - n^{1/p}/2}{r_p+n^{1/p}/2}\right)^n,
\end{equation}
where (a) is due to Lemma \ref{thm:radiusAsymptotic} and the proof of Lemma \ref{prop:detNumberPoints}. But the right-hand side of Inequality (\ref{eq:limitRadius}) cannot exceed $\Delta_p^n$, giving us the bound.
The asymptotic part follows from the fact that ${\Delta_p^n}$ goes to zero exponentially fast, for fixed $p$ and large $n$ (see, e.g., \cite[Thm. 2, p. 415]{Leker} for a proof that ${\Delta_p^{n} \leq 2^{-n/p + \log(n/p+1)}}$, for $p > 2$).
\end{proof}

Note that the bound ${\Delta_p^{n} \leq 2^{-n/p + \log(n/p+1)}}$ \cite[Thm. 2, p.415]{Leker} is non-trivial for large enough $n$ (more precisely, when $n$ is such that $(n/p+1) < 2^{n/p}$). In this case, a bound on the packing radius of a perfect code independent on $\Delta_p^{n}$ can be obtained:

\begin{equation} r_p \leq \frac{n^{1/p}}{2} \frac{\left(1+(\Delta_p^n)^{1/n}\right)}{\left(1-(\Delta_p^n)^{1/n}\right)} \leq \frac{n^{1/p}}{2} \left(\frac{2^{1/p} + (1 + n/p)^{1/n}}{2^{1/p} - (1+n/p)^{1/n}} \right) \approx \frac{n^{1/p}}{2}.
\end{equation}

The table below shows numerical values for Corollary \ref{cor:nonExistenceAsymptotic} for $p=2$, using the known best packing densities for lattices in dimensions $n = 1, \ldots 8$ and $n = 24$ \cite{SloaneLivro}.
\begin{table}[!htb]$$
\begin{array}{|c|c|c|c|c|c|c|c|c|}
\hline
\mbox{n} & 2 & 3 & 4 & 5 & 6 & 7 & 8 & 24 \\
 \hline
\bar{r}_{n,2} & \sqrt{838} & \sqrt{299} & \sqrt{274} & \sqrt{214} & \sqrt{223} & \sqrt{231} & \sqrt{273} & \sqrt{357} \\
 \hline
\end{array}$$
\caption{For $p=2$, dimension $n$ versus threshold radius $\overline{r}_{n,2}$ that guarantees no linear $r$-perfect codes in $\mathbb{Z}^2$ for $r \geq \overline{r}_{n,2}$}
\label{tab:tab}
\end{table}

\section{Case Study: $p = 2$}
\label{sec:casestudy}
A case of interest is when $p = 2$. In this case, we have clear bounds on the packing density. We also remark that, codes in the $\ell_2$ metric in $\mathbb{Z}_q^n$ were recently employed by Belfiore and Sol\'e in \cite{BelfioreSole} as inner codes for some constructions of spherical codes. They conjecture that perfect codes for $p = 2$ will provide stronger constructions, and pose the existence of such codes as an open question. We use Theorem \ref{thm:nonExistenceAsymptotic} to provide negative answers when $n = 2$ and $3$, for $2r_2(\mathscr{C}) < q$. We will use the following result by Horak and AlBdaiwi \cite{Diameter}:

\begin{theorem}{\upshape{\cite[Thm. 6]{Diameter}}}
Let $\mathcal{P} \subset \mathbb{Z}^n$, such that $|\mathcal{P}| = m$. There is a lattice tiling of $\mathbb{Z}^n$ by translates of $\mathcal{P}$ if and only if there is an Abelian group $G$ of order $m$ and a homomorphism $\phi: \mathbb{Z}^n \to G$ such that the restriction of $\phi$ to $\mathcal{P}$ is a bijection.
\label{thm:Homomorphism}
\end{theorem}

\begin{theorem}\label{thm:dimension2} There are no linear $r$-perfect codes in $\mathbb{Z}^2$ in the $\ell_2$ metric, unless $r = 1, \sqrt{2}, 2,$ or $2\sqrt{2}$.
\end{theorem}
\begin{proof}
(i) Asymptotic Part: The best sphere packing lattice in $\mathbb{R}^2$ has density $\pi/\sqrt{12}$. From Corollary \ref{cor:nonExistenceAsymptotic} there are no perfect codes for $r \geq 29$. In fact, we can tighten this estimate by explicitly calculating the ratio
$$\frac{V_{n,2} (r - \sqrt{2}/2)^2}{\mu_2(n,r)} = \frac{\pi (r - \sqrt{2}/2)^2}{\mu_2(n,r)}.$$
for all critical radii $r \leq  \sqrt{838}$. From this we conclude that, the packing density of the induced sphere packing will exceed $\pi/\sqrt{12}$ for $r \geq \sqrt{294}$.


\noindent (ii) We use Theorem \ref{thm:Homomorphism} with $\mathcal{P} = B_2^2(r)$. The existence of possible homomorphisms for critical radii $r\leq \sqrt{294}$ was checked trough an algorithm implemented in Wolfram Mathematica which made use of  Proposition  \ref{thm:distanceSet} to discard a meaningful number of cases to be tested. If $|B_2^{2}(r)|$ is square free only the cyclic group needs to be considered otherwise other possible  Abelian groups  were checked for the existence  of homomorphisms.  From these we can conclude that no such homomorphism exist, unless $r = 1, \sqrt{2},2 $ or $2 \sqrt{2}$. Examples of homomorphisms in these cases are $\phi(x,y) = x+\overline{2}y \in \mathbb{Z}_5,$  $\phi(x,y) = x+\overline{3}y \in \mathbb{Z}_9$, $\phi(x,y) = x+\overline{5}y \in \mathbb{Z}_{13}$ and $\phi(x,y) = x+\overline{5}y \in \mathbb{Z}_{25}$. The perfect codes associated to these homomorphisms as their kernels are the lattices with bases $\{(1,2),(0,5)\}$,
$\{(3,2), (0,3)\}$ , $\{(1,5)),(3,2)\}$ and $\{(5,4),(0,5)\}$, respectively. We can also view these lattices as if they were obtained by Construction A from the perfect codes $\left\langle \left(\overline{1},\overline{2}\right) \right\rangle \subset \mathbb{Z}_{5}^{2}$,  $\left\langle \left(\overline{3},\overline{2}\right) \right\rangle \subset \mathbb{Z}_9^{2}$, $\left\langle \left(\overline{1},\overline{5}\right)\right\rangle \subset \mathbb{Z}_{13}^{2}$ (See Figure 3)  and $\left\langle \left(\overline{5},\overline{4}\right)\right\rangle \subset \mathbb{Z}_{25}^{2}$ in the $\ell_2$ metric. 

\begin{figure}[!htb]
\centering
\includegraphics[scale=0.4]{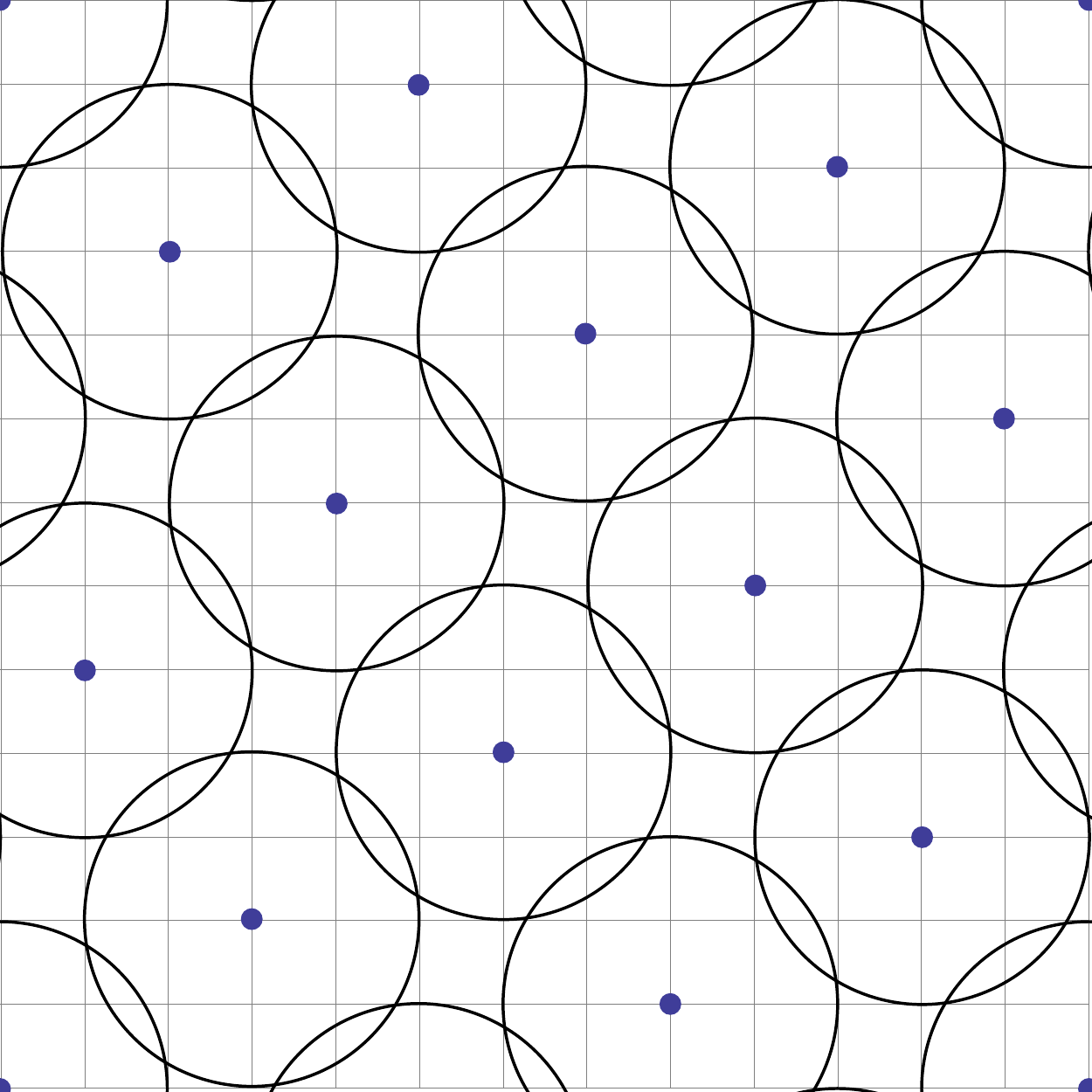}
\caption{Packing balls for the code $\mathscr{C} = \left\langle
\left(\overline{1},\overline{5}\right)\right\rangle \subset \mathbb{Z}_{13}^{2}$ in the $\ell_2$ metric.}
\label{fig:perfectl2}
\end{figure}
\end{proof}

\begin{remark}  Since $B_2^{2}(1) = B_p^{2}(1)$ and $B_2^{2}(2) = B_p^{2}(2)$ for any $p$, $1 \leq p < \infty$, for $r=1$ and $r=2$ perfect codes in $\mathbb{Z}^{2}$ in the $\ell_2$ metric, as the one listed in the last theorem, are also perfect in the $\ell_p$ metric for any $p$. Note also that $B_{2}^{2}(\sqrt{2}) = B_{\infty}^{2}(1)$  and $B_{2}^{2}(2\sqrt{2}) = B_{\infty}^{2}(2)$  and then the perfect codes in the $\ell_2$ metric with radii  $\sqrt{2}$ and $2\sqrt{2}$ of the last theorem as well as the trivial codes $3\mathbb{Z}^{2}$  and $5\mathbb{Z}^{2}$  are perfect codes in $\mathbb{Z}^{2}$  in the $\ell_{\infty}$ metric with radii $1$ and $2$, respectively (Corollary \ref{cor:perfectinfty}).  	
	
\end{remark}

\begin{theorem}\label{thm:dimension3} There are no linear $r$-perfect codes in $\mathbb{Z}^3$ in the $\ell_2$ metric, unless $r = 1,$ or $\sqrt{3}$.
\end{theorem}
\begin{proof} Again, by using the same arguments of the previous theorem, we can tighten the estimate provided by Table \ref{tab:tab}, by explicitly calculating
$$\frac{V_{3,2} (r - \sqrt{3}/2)^2}{\mu_2(3,r)} = \frac{(4/3) \pi (r - \sqrt{3}/2)^3}{\mu_2(3,r)},$$
for $r\leq \sqrt{299}$ and ensuring that it does not exceed the best Euclidean packing density in $\mathbb{R}^{3}$, $\Delta_{2}^{3} = \frac{\pi}{3\sqrt{2}}$. This gives us $r \leq \sqrt{92}$. We consider Theorem \ref{thm:Homomorphism} with $\mathcal{P}=B_2^{3}(r)$ and Proposition \ref{thm:distanceSet} to discard a meaningful number of cases to be tested. An algorithm in Wolfram Mathematica was used taking this into consideration to check all possible homomorphisms.  We conclude then that the only possible critical radii for perfect codes in $\mathbb{Z}^{3}$ in the $\ell_2$ metric  are $1$ and $\sqrt{3}$. Examples of homorphisms in these cases are $\phi(x,y,z) = x+ \overline{2}y+\overline{3}z \in \mathbb{Z}_{7}$ and $\phi(x,y,z) = x+ \overline{3}y + \overline{9}z \in \mathbb{Z}_{27}$ 
and and their kernels provide the  associated perfect  codes   given by the lattices with basis $\alpha=\{(1,0,2),(0,1,4),(0,0,7)\}$ and $\beta=\{(3,8,0),(0,3,2),(0,0,3)\}$, respectively. We can also view  these lattices as if they were obtained by Construction A from the perfect codes $\left\langle \left(\overline{1},\overline{0},\overline{2}\right),\left(\overline{0},\overline{1},\overline{4}\right) \right\rangle \subset \mathbb{Z}_{7}^{3}$ and  $\left\langle \left(\overline{3},\overline{8},\overline{0}\right), \left(\overline{0},\overline{3},\overline{2}\right) \right\rangle \subset \mathbb{Z}_{27}^{3}$ in the $\ell_2$ metric. 
\end{proof}

\begin{remark}  Since $B_{1}^{n}(1) = B_p^{n}(1)$ for any $n$ and $p$, $1 \leq  p < \infty$, for $r=1$ the lattice with basis $\alpha$ in the last proposition is also $1$-perfect for any $p$, $1\leq p< \infty$. On the other hand  any perfect code in $\mathbb{Z}^{3}$  in the $\ell_2$ metric and  packing radius $\sqrt{3}$, as the one given in the last proposition and also by $3\mathbb{Z}^{3}$, is also perfect in the $\ell_{\infty}$ metric  with packing radius $1$ (Corollary \ref{cor:perfectinfty}). \end{remark}

For $n \geq 4$, the number of cases that need to be checked is increasingly large. We thus believe that new algebraic techniques should be used to establish the non-existence of perfect codes. Furthermore, from Proposition \ref{prop:cubicPolyominoes}, cubic polyominoes do not exist in the $\ell_2$ metric for $n \geq 4$ since for all integer $r>0$, the condition $n r^{2} < (r+1)^2$ implies $n <4$. Since these are the only cases found computationally, we believe that the following conjecture is true:

\begin{conjecture} There are no linear perfect codes with parameters $(n,r,2)$ except if $r = 1$ or $(n,r) \in \left\{ (2,\sqrt{2}), (2,2), (2, 2\sqrt{2}), (3,\sqrt{3}) \right\}$. In particular, there are no linear perfect codes in the $\ell_2$ metric if $n > 3$ and $r > 1$.

\end{conjecture}


In the next section we get some quantitatively weaker results but that also hold for non linear codes and may provide some insights on the tiling problem.

\section{Non-Existence of Tilings of Special Formats}\label{no}

When $n$, $p$ and $r$ vary, the polyomino $T_p^{n}(r)$ may change its shape. In this section we prove ``from scratch'' that there is no tiling of $\mathbb{R}^2$ by $T_p^{2}(r)$  when $r>2$ is an integer and $1 < p < \infty$ satisfies $(r-1)^{p}+2^{p} \leq r^{p}$. This proof holds even for non-linear codes (non-lattice tilings) and for all $p\geq 2$.

\begin{figure}[!htb] \label{fig:polyointeger} \[\begin{array}{cccc}
\includegraphics[scale=0.3]{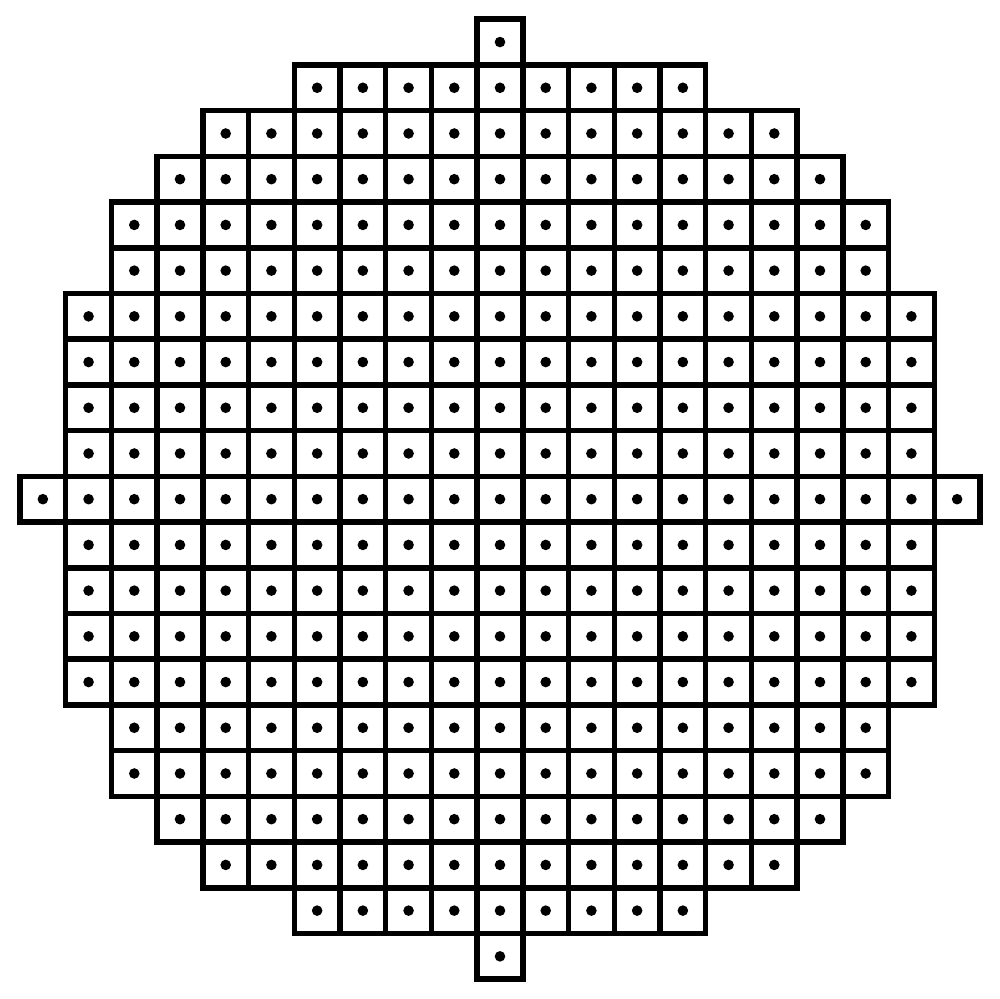} & \includegraphics[scale=0.3]{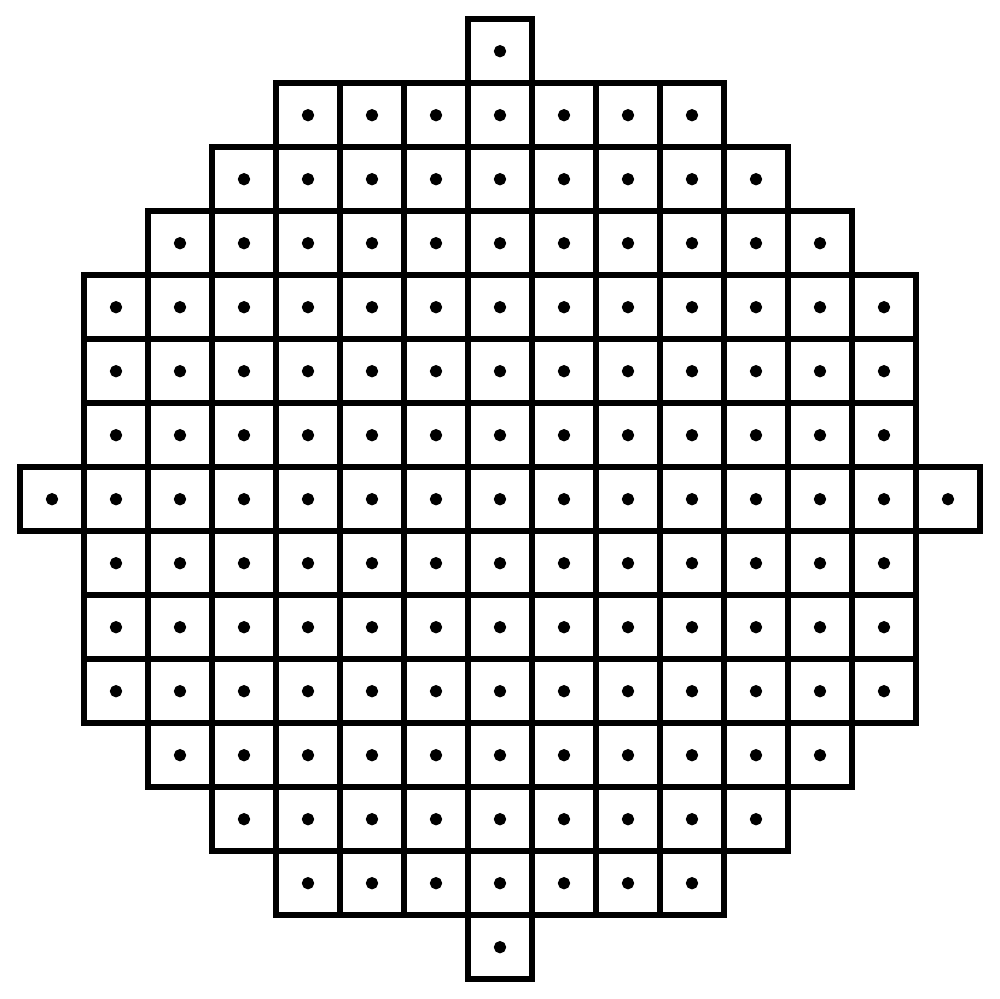} & \includegraphics[scale=0.15]{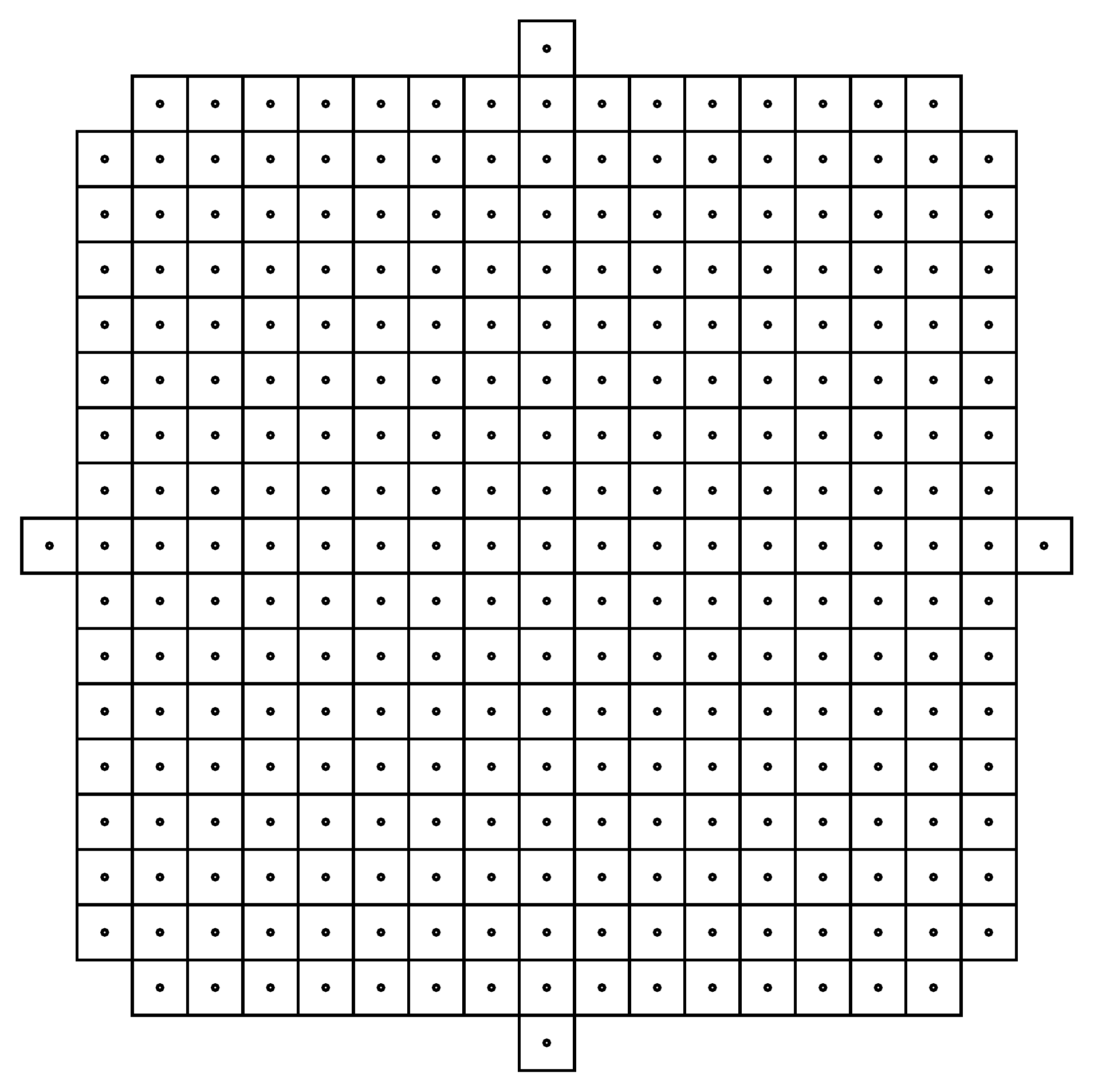} & \includegraphics[scale=0.15]{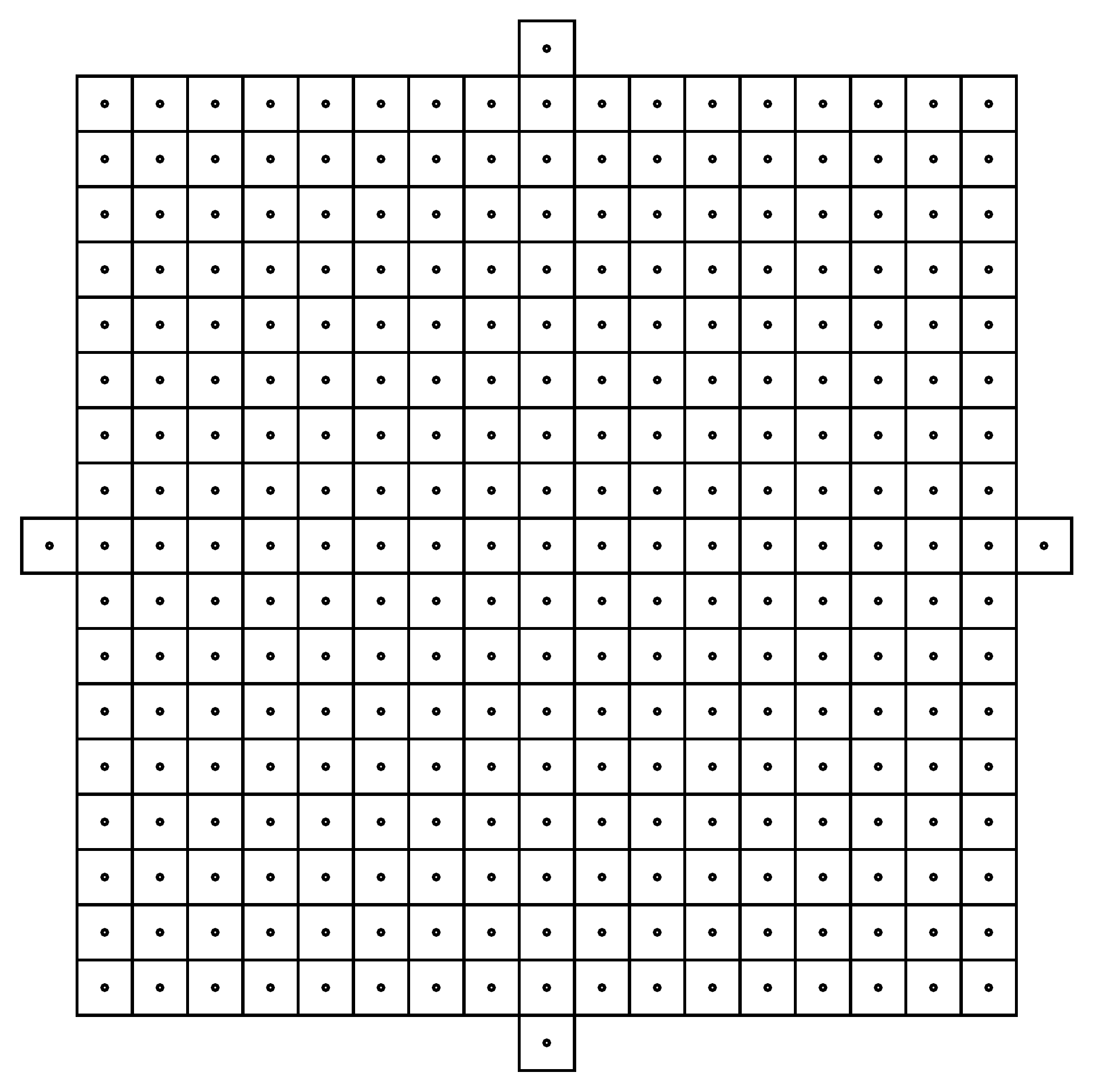}
\\ T_{2}^{2}(10) & T_{2}^{2}(7) & T_{4}^{2}(10) & T_{6}^{2}(10)
\end{array}\]
\caption{Polyominoes $T_p^{n}(r)$}
\end{figure}

Let $r$ be a positive integer and suppose that there is a tiling of $\mathbb{R}^{n}$ by translates  of $T_p^{n}(r)$. Let $\mathscr{C}$ be the set associated to the translations of such a tilling. From \cite[Thm. 4, p. 36]{Szabo} we may assume without loss of generality $\mathscr{C} \subset \mathbb{Z}^{n}.$ Suppose that $\cc_1 = (0,0,\ldots,0) \in \mathscr{C}$ (i.e., that the polyomino centered at the origin $T_p^{n}(r)$ is part of the tiling). 

For $\cc \in \mathscr{C}$ we call $\xx=\cc \pm r \bm{e}_i \in \cc + T_p^{n}(r)$ an \textbf{endpoint} since $\cc \pm r\bm{e}_i  \pm \bm{e}_j \not\in T_p^{n}(r)$ if $j\neq i$. Otherwise $\xx \in T_p^{n}(r)$ is called an \textbf{ordinary point}. 

Once fixed $n=2$, $r$ a positive integer and $1 < p < \infty$ such that $(r-1)^{p} + 2^{p} \leq r^{p}$, let $T = T_p^{2}(r) = T_1$.  

Note that if $r>2$, $\cc  \pm (r-1)\bm{e}_i \pm h\bm{e}_j \in T_p^{2}(r)$ for $j\neq i$ and $h=1,2$. This shape in the surrounding of endpoints will imply that there no tiling of $\mathbb{R}^{2}$ by $T_p^{2}(r)$ as shown next.


We have that  $(r,0)$ is an endpoint of $T_1$. Consider $\yy_1 = (r,-1) \not\in T_1$. Then $\yy_1$ is either an endpoint or an ordinary point for another tile, say, $T_2$. The technique is to exhibit points that cannot be covered by any tile without overlap with one of the preceding tiles.

First observe that the general lemma holds for tilings by translates of $T_1$.

\begin{lemma}[Opposite Endpoints] Let $\xx \in  \cc + T$ be an endpoint such that $\xx = \cc + r \ee_i$. Suppose that $\xx \pm \ee_j$, $j \neq i$, is an endpoint of another tile $\tilde{T} =  \tilde{\cc}+ T$. Then $\xx \pm \ee_j = \tilde{\cc} -r\ee_i$. Similarly, if $\xx = \cc  -r \ee_i$,  then $\xx \pm \ee_j = \tilde{\cc} + r\ee_i$.
\end{lemma}
\begin{proof} It is sufficient to prove the lemma for $\cc = \bm{0}$. If $\xx+\ee_j$ is an endpoint of $\tilde{\cc}+ T$, then $\xx+\ee_j = \tilde{\cc} \pm r \ee_k$. We need to prove that $k = i$  and then that $\xx+\ee_j = \tilde{\cc} - r \ee_i$ . If $\xx+\ee_j = \tilde{\cc} + r \ee_k$, then $\tilde{\cc}= r \ee_i + \ee_j - r \ee_k$ and the point $\tilde{\yy} = \tilde{\cc} + (r-1) \ee_k - \ee_i \in T \cap \tilde{T}$, which is a contradiction. On the other hand, if $\xx+\ee_j = \tilde{\cc}  -r \ee_k$, then $\tilde{\cc} = r \ee_i +\ee_j+ r\ee_k$ and $\tilde{\yy} = \tilde{\cc} -(r-1)\ee_k - \ee_i \in T \cap \tilde{T}$, unless $k = i$. The other cases are analogous.
\end{proof}

\begin{theorem}\label{z2integer} Let $r$ be an integer and $1 < p < \infty$ such that $(r-1)^{p}+2^{p} \leq r^{p}$. It is not possible to tile $\mathbb{R}^{2}$ with translates of $T_p^{2}(r)$, unless $r \leq 2$.
\end{theorem}
\begin{proof}  If $\yy_1$ is an endpoint of $T_2 = \cc_2 + T$, then, from the opposite endpoints property, $\yy_1 = \cc_2 -r\ee_1$ and $\cc_2 = (2r,-1)$. Now let $\yy_2 = (r,-2) \in T_3$, where $T_3 = \cc_3+ T$ is a third tile. We have two cases for considering:\\
(i) If $\yy_2$ is an endpoint, then, again from the opposite endpoints, $\yy_2 = \cc_3 + re_1$ and $\cc_3 = (0,-2) \in T_1 \cap T_3$.\\
(ii) If $\yy_2$ is an ordinary point, $\yy_2 = \cc_3 + (x_1,x_2)$ where $(x_1,x_2)$ is an ordinary point of $T$ and $\cc_3=(r-x_1,-2-x_2)$. If $x_1<0$, then $(x_1+1,x_2) \in T$ and then $\yy_2 + \ee_1 = (r+1,-2) \in T_2 \cap T_3$.  If $x_1 \geq 0$, then $(x_1-1,x_2) \in T$ and $\yy_2-\ee_1 = (r-1,-2) \in T_1 \cap T_3$. \\
\indent Now, if  $\yy_1 = (r,-1)$ is an ordinary point of a tile $T_2$, it follows that $\yy_1 = \cc_2 + (x_1,x_2)$, where $\cc_2 = (r-x_1, -1-x_2) \in \mathscr{C}$ and $(x_1,x_2)$ is an ordinary point of $T$.
 Consider $\yy_2 = (r,1) \not\in T_1 \cup T_2.$ \\
(i) If $\yy_2$ is an endpoint of a tile $T_3$, then, from the opposite endpoints property, $\yy_2 = \cc_3 -r \ee_1$ and $\cc_3 = (2r, 1) \in \mathscr{C}$. If $\yy_3 = (r+1,-1)$, then $\yy_3 \in T_2 \cap T_3$. Indeed,  $\yy_3 \in T_2$ since the fact that $\yy_1$ is an ordinary point and $\yy_1 - \ee_1 = (r-1,-1) \in T_1$ implies that $\yy_1 + \ee_1 = \yy_3 \in T_2$. Also $\yy_3= \cc_3 + (-(r-1),-2) \in T_3$.\\
(ii) If $\yy_2$ is an ordinary point of a tile $T_3$, then $\yy_2 = \cc_3 + (\tilde{x_1},\tilde{x_2})$, where $\cc_3 = (r-\tilde{x_1}, 1-\tilde{x_2}) \in \mathscr{C}$ and $(\tilde{x_1},\tilde{x_2})$ is an ordinary point of $T$. Both $x_1$ and $\tilde{x_1}$ are negative. Indeed, if $x_1 \geq 0$, then $(x_1-1,x_2) \in T$ and  $\yy_1 - \ee_1 \in T_1 \cap T_2$.  By a similar argument, if $\tilde{x_1} \geq 0$, then $y_2-e_1 \in T_1 \cap T_3$. We also have $\tilde{x_1} \leq -2$, otherwise, if $ \tilde{x_1} = -1$, $(\tilde{x}_1-1,\tilde{x_2}) \in T$ and $\yy_2 - \ee_1 = \cc_3 + (\tilde{x}_1,\tilde{x}_2) + (-1,0) \in T_3 \cap T_1$. If $-\tilde{x_1}=\min\{-\tilde{x_1},-x_1\}$, then $\yy_3 = (r-\tilde{x_1},0) \in T_3$. \\ 
\noindent (a) If $\yy_4 = (r+1,0) \in T_2 \backslash T_3$, then $\yy_3 = (r -\tilde{x_1},0) \in T_2 \cap T_3$. Indeed, since $-x_1 \geq -\tilde{x_1}$, then $\yy_3 = \yy_4 -\tilde{x_1}\ee_1 -\ee_1 \in T_2$.  \\ 
\noindent (b) If $\yy_4=(r+1,0) \in T_3 \backslash T_2$, then $\yy_5 = (r -\tilde{x_1},-1) \in T_2 \cap T_3$. Indeed, since $-x_1 \geq -\tilde{x_1}$, then $(r,-1)+ (-\tilde{x_1},0) = \yy_5 \in T_2$. Now, note that $\yy_3$ is not an endpoint. Indeed,  $\yy_3 + \ee_2 \in T_3$ and since $\yy_4 = (r+1,0) \in T_3$, $\yy_3 = (r - \tilde{x_1},0) \in T_3$ and $-\tilde{x_1} \geq 2$, it follows that $(r+s,0) \in T_3$ for $s=1,2,\ldots, -\tilde{x_1}$. Since $\yy_3$ is not an endpoint and $\yy_3 = \cc_3 + (0,1-\tilde{x}_2)$ it follows that $\yy_3 - \ee_2 = \yy_5 \in T_3$. \\
\noindent (c) If $\yy_4 = (r+1,0) \not\in T_2 \cup T_3$, then this point does not belong to any other tile $T_4$. Indeed, since $\yy_4 + \ee_2 = (r+1,1) \in T_3$ and $\yy_4 - \ee_2 = (r+1,-1) \in T_2$,  $\yy_4$ must be an endpoint of the form $\yy_4 = \cc_4 - r \ee_2$, where $\cc_4 \in \mathscr{C},$ and then $\yy_4 + \ee_1 + \ee_2 \in T_3 \cap T_4$. 
\end{proof}

This technique ``from scratch'' becames harder in dimensions greater than $2$. The next theorem, proved in Appendix A, is a limit case for general dimension.

\begin{theorem}\label{thm:tiledimensionn} Let $r$ be an integer and $n \geq3$. If $(n-1)(r-1)^{p} + (r-2)^{p} \leq r^{p}$, then it is not possible to tile $\mathbb{R}^{n}$ with translates of $T_p^{n}(r)$, unless $r \leq 2$.
\end{theorem}

Based on the results obtained here we close this paper with an open question: \vspace{0.2cm}

Must a perfect code in $\mathbb{Z}^{n}$ in the $\ell_p$ metric, $2 \leq p < \infty$,  be either  perfect in the $\ell_1$ metric  (Lee metric) or in  the $\ell_{\infty}$ metric?  

\section{Acknowledgments}
The authors would like to thank the reviewers for their comments and suggestions.

\section*{Appendix A}

In Section \ref{no} we have shown that is not possible to tile $\mathbb{R}^{2}$ with translates of the polyomino $T_{p}^{2}(r)$ if $r\geq 3$ is an integer and $1<p<\infty$ satisfies $(r-1)^{p}+2^{p}\leq r^{p}$. We show next that if $(n-1)(r-1)^{p} + (r-2)^{p} \leq r^{p}$, then it is not possible to tile $\mathbb{R}^{n}$ with translates of $T_p^{n}(r)$ for $r\geq 3$ integer and $n \geq 3$. \vspace{0.5cm} 

%
%

{\bf Proof of the Theorem \ref{thm:tiledimensionn}}. \\

Once fixed $n \geq 3$, $r \geq 3$ an integer and $1 < p < \infty$ with $(n-1)(r-1)^{p} + (r-2)^{p} \leq r^{p}$, let $T = T_{p}^{n}(r) = T_1$.

Suppose that there is a tiling of $\mathbb{R}^n$ by translates of $T$ given by elements of  $\mathscr{C} \subset \mathbb{Z}^n$ and that $(0,0,\ldots,0) \in \mathscr{C}.$

We have that  $(r,0,\ldots,0)$ is an endpoint of $T_1$. Consider $\yy_1 = (r,-1,0,\ldots,0) \not\in T_1$. Then $\yy_1$ is either an endpoint or an ordinary point for another tile, say, $T_2$. By direct generalization of the first part of the proof of Theorem \ref{z2integer}, the  case where $\yy_1$ is an endpoint is discarded. 

To discard the remaining case, let $\yy_1 = (r,-1,0,\ldots,0) \in T_2 = \cc_2 + T$ be an ordinary point. Since $\yy_1$ is an ordinary point of $T_2$, it follows that $\yy_1 = \cc_2 + (x_1,\ldots,x_n)$ where $(x_1,\ldots,x_n)$ is an ordinary point of $T$ with $x_1 \leq -r+2$ and $x_2 \geq r-2$. Indeed, if $x_1 > -r+2$ then  $(x_1-1,\ldots,x_n) \in T$ and $\yy_2 = \yy_1 -\ee_1 = (r-1,-1,0,\ldots,0) \in T_1 \cap T_2$ and if  $x_2 < r-2$, then  $(x_1,x_2 +1 \ldots,x_n) \in T$ and $\yy_3 = \yy_1 + \ee_2 = (r,0,0,\ldots,0) \in T_1 \cap T_2$.

From $\yy_1 = \cc_2 + (x_1,\ldots,x_n)$, it follows that $\cc_2 = (r-x_1,-1-x_2,-x_3,\ldots,-x_n)$. Consider $\yy_2 = \cc_2 + r \ee_2 - \ee_1 = (r-1-x_1, -x_2 + r-1, -x_3,\ldots,-x_n) \not\in T_1 \cup T_2$.  There are two cases for considering: \\
\indent (i) If $\yy_2 = (r-1-x_1, -x_2+r-1, -x_3,\ldots,-x_n)$ is an endpoint of a tile $T_3$, then, from the opposite endpoints property, $\yy_2 = \cc_3 - r \bm{e}_2$, where $\cc_3 \in \mathscr{C}$ and $\cc_3 = (r-x_1-1,-1-x_2+ 2r,-x_3,\ldots,-x_n)$. If $x_3 \geq 0$ and $\yy_3 = \yy_2 + x_1 \ee_1 + \ee_2 + \ee_3 =  (r-1, r-x_2, -x_3+1,\ldots,-x_n)$, then $\yy_3 \in T_1 \cap T_3$.  If  $x_3 < 0$ and $\yy_3 = \yy_2 + x_1 \ee_1 + \ee_2 - \ee_3 =  (r-1, r-x_2, -x_3-1,\ldots,-x_n)$, then $\yy_3 \in T_1 \cap T_3$. \\
\indent (ii) If $\yy_2 = (r-1-x_1, -x_2+r-1, -x_3,\ldots,-x_n)$ is an ordinary point of $T_3$, then $\yy_2 = \cc_3 + (\tilde{x_1},\ldots,\tilde{x_n})$, where $\cc_3 = (r-1-x_1-\tilde{x_1},-x_2+r-1-\tilde{x_2},-x_3-\tilde{x_3},\ldots,-x_n-\tilde{x_n}) \in \mathscr{C}$, $(\tilde{x_1},\ldots,\tilde{x_n})$ is an ordinary point of $T$ and $\tilde{x_1} \geq r-2$. Indeed, if $\tilde{x_1} < r-2$, then $(\tilde{x_1}+1,\ldots,\tilde{x_n}) \in T$ and $\yy_2 + \ee_1 \in T_2 \cap T_3$. If $\yy_4 = \yy_2 -  2\tilde{x_1} \ee_1  = (r -1- x_1 - 2\tilde{x_1}, -x_2 +r-1, -x_3,\ldots,-x_n)$, then  $\yy_4 \in T_1\cap T_3$.

Since $\yy_1$ is neither an endpoint nor an ordinary point, the result follows  by contradiction.  \\

\bibliography{arxiv_Version}
\bibliographystyle{alpha}





%
%
%




%
%
%
\end{document}